\documentclass[a4paper,leqno,twoside]{amsart}%
\usepackage{amsmath,amssymb,amsthm,enumerate,hyperref,color}
\numberwithin{equation}{section}
\newtheorem{theorem}{Theorem}[section]

\newtheorem{proposition}[theorem]{Proposition}
\newtheorem{lemma}[theorem]{Lemma}
\newtheorem{corollary}[theorem]{Corollary}
\newtheorem{remark}[theorem]{Remark}

\newcommand{\rad}{{\text{\upshape rad}}}
\newcommand{\loc}{{\text{\upshape loc}}}
\newcommand{\nod}{{\text{\upshape nod}}}

\renewcommand{\a}{\alpha}

\def\wtu{\widetilde u}

\def\e{{\varepsilon}}

\def\L{{\Lambda}}
\def\l{{\lambda}}
\def\a{{\alpha}}
\def\b{{\beta}}

\newcommand{\jcal}{{\mathcal J}}
\newcommand{\n}{N}
\newcommand{\R}{{\mathbb R}}
\newcommand{\N}{{\mathbb N}}

\def\sideremark#1{\ifvmode\leavevmode\fi\vadjust{\vbox to0pt{\vss
 \hbox to 0pt{\hskip\hsize\hskip1em
 \vbox{\hsize2.1cm\tiny\raggedright\pretolerance10000
  \noindent #1\hfill}\hss}\vbox to15pt{\vfil}\vss}}}%
\newcommand{\edz}[1]{\sideremark{#1}}

\definecolor{darkgreen}{rgb}{0.0, 0.5, 0.2}
\definecolor{purple}{rgb}{0.5, 0.0, 0.5}
\newcommand{\AL}{\color{purple}}

\newif\ifcomment \commentfalse
\def\commentON{\commenttrue}

\long\outer\def\BC#1\EC{\ifcomment \sloppy \par \# \ldots\dotfill
	{\em #1} \dotfill \# \par \fi } \commentON

\newcommand{\remove}[1]{}

\title{On the asymptotically linear H\'enon problem} 
\author[A.~L.~Amadori]{Anna Lisa Amadori$^\dag$}
\thanks{The author is member of the Gruppo Nazionale per l'Analisi Matematica, la Probabilit\`a e le loro Applicazioni (GNAMPA) of the Istituto Nazionale di Alta Matematica (INdAM). }
\date{\today}
\address{$\dag$ Dipartimento di Scienze Applicate, Universit\`a di Napoli ``Parthenope", Centro Direzionale di Napoli, Isola C4, 80143 Napoli, Italy. \texttt{annalisa.amadori@uniparthenope.it}}

\begin{document}
\maketitle 

\begin{abstract}
	In this paper we consider the H\'enon problem in the ball with Dirichlet boundary conditions. We study the asymptotic profile of radial solutions and then deduce the exact computation of their Morse index when the exponent $p$ is close to $1$.
	Next we focus on the planar case and describe the asymptotic profile of some  solutions which minimize the energy among functions  which are invariant for reflection and rotations of a given angle $2\pi/n$. By considerations based on the Morse index we see that, depending on the values of $\a$ and $n$, such least energy solutions can   be radial, or nonradial and different one from another.
	\end{abstract}

\



\section{Introduction}
This paper investigates the H\'enon problem
\begin{equation} \label{H}
\left\{\begin{array}{ll}
-\Delta u = |x|^{\alpha}|u|^{p-1} u \qquad & \text{ in } B, \\
u= 0 & \text{ on } \partial B,
\end{array} \right.
\end{equation} 
where $\a> 0$, $B$ stands for the unitary ball in $\R^N$ with $N\ge 2$, and the exponent $p$ is next to one.

It is well known that, for $\a>0$ fixed, the H\'enon problem \eqref{H} admits solutions, and in particular radial solutions, for every $p> 1$ in dimension $N=2$, and for every $p\in(1,\frac{N+2+2\a}{N-2})$ in dimension $N\ge 3$.
In that range of existence, for any given $m\ge 1$ there  is exactly one couple of radial solutions of \eqref{H} which have exactly $m$ nodal zones, they are classical solutions and they are one the opposite of the other (see \cite{Ni}, \cite{BWi}, \cite{NN}). 
It is well known that 
\eqref{H} has also nonradial positive solutions, and the literature on this subject is rich. First   \cite{SSW} showed that  for every value of $p$ below the Sobolev critical exponent, then the minimal energy solution is nonradial when $\a$ is large enough.
Next  nonradial solutions  have been produced in a number of works by various techniques: we mention among others \cite{PS}, \cite{EPW},  \cite{HCZ} based on  Lyapunov-Schmidt reduction, \cite{S}, \cite{AG-N=2} relying on different constrained minimization,  and  \cite{AG14} using bifurcation.
Nevertheless it is worth mentioning that  for $\a>0$ fixed there is a neighborhood of $p=1$ (clearly depending by $\a$) where the only positive solution to \eqref{H} is the radial one, see \cite[Theorem 3.1]{AG14}.
\\
Concerning nodal solutions, considerations based on the Morse index yield that the minimal energy solution is nonradial for every $\a>0$ and $p$ in the existence range. Indeed the minimal energy nodal solution has Morse index 2 by \cite{BW}, while the Morse index of nodal radial solutions is greater than 4, see \cite{AG-sez2-bis}. The same phenomenon  was already pointed out in \cite{AP}, concerning the Lane Emden problem.
In dimension $N=2$ and for large values of $p$ sign-changing multipeak solutions have been produced in \cite{ZY}.
In the same range  \cite{AG-N=2} proved a multiplicity result by constructing least energy nodal solutions in suitable symmetric spaces and comparing their Morse index with the one of radial solutions.
In higher dimension another recent paper  \cite{KW}  produced  nonradial solutions  by bifurcation w.r.t.~the parameter $\a$, by taking advantage from the fact that the Morse index of radial solutions goes to infinity.
In this perspective knowing the exact Morse index is an essential step in producing nonradial solutions. 
Its value has been computed when the parameter $\a$ is large  in \cite{LWZ}, and when the parameter $p$ is close to the supremum of the existence range in \cite{AG-N>3} (in dimension $N\ge 3$) and in \cite{AG-N=2} (in dimension $N=2$).
\\
Here we describe the asymptotic profile of radial solutions and compute the exact value of their Morse index for $p$ close to $1$, and we use it to obtain, for any given value of $\a$, nonradial solutions which live in a  neighborhood of $p=1$, and which anyway preserve some rotational symmetry. In doing this we also describe the asymptotic profile of such noradial solutions. 

But let us go in order.
The papers \cite{BBGvS} and \cite{Gro09} investigated  the Lane-Emden problem settled in any domain $\Omega$, when $p$ approaches 1,
and described the behaviour of the solutions in terms of the eigenvalues of the Laplace operator on $\Omega$, i.e. 
\begin{equation}\label{prima-autof}
\begin{cases}
-\Delta \omega = \mu  \, \omega \quad & \text{ in } \Omega , \\
\omega= 0 & \text{ on } \partial\Omega ,
\end{cases}
\end{equation}
showing, among other things, that any sequence of solutions  whose norm in $L^2(\Omega)$ is suitably bounded converges (up to a subsequence) to an eigenfunction of \eqref{prima-autof}.
The radial setting is clearly simpler, and one can see that any radial solutions satisfies the $L^2$ bound as $p$ is next to 1, and a sequence of radial solutions with $m$ nodal zones converges to the $m^{th}$ radial eigenfunction of \eqref{prima-autof}, which is simple and is nothing else that a Bessel function.
Analogous result holds for the H\'enon problem \eqref{H}, provided that the eigenvalue problem for the Laplacian is replaced by the weighted eigenvalue problem
\begin{equation}\label{prima-autof-weight} 
\left\{
\begin{array}{ll}
-\Delta \omega= \mu |x|^{\a} \omega & \text{ in } B, \\
\omega = 0 & \text{ on } \partial B ,  
\end{array}\right.
\end{equation}
which clearly reduces to \eqref{prima-autof} when $\a=0$ and $\Omega$ is a ball.
Our first result, presented in Section \ref{S:1}, stands in computing the eigenvalues and eigenfunctions of \eqref{prima-autof-weight} and describing the asymptotic behaviour of radial solutions as $p\to 1$.
\\
We denote by $\Gamma$  the Gamma-function, by $\jcal_{\beta}$ the Bessel function of first kind defined as 
\[\jcal_{\beta}(r) = r^{\beta}\sum\limits_{k=0}^{+\infty} \dfrac{(-1)^k}{k!\Gamma(k+1+\beta)} \left(\frac{r}{2}\right)^{2k}, \quad r\ge 0 , \]
and by $z_n(\beta)$ the sequence of its positive zeros.
We will prove that

\begin{theorem}\label{p1}
		Let $u_p$ be a radial solution to \eqref{H} with $m$ nodal zones for $\a\ge 0$.
	When  $p\to 1$ we have  
	\begin{align}\label{a0}
	\| u_p\|_{\infty}^{\frac{p-1}{2}} & \to \frac{2+\a}{2} z_m\left(\frac{\n-2}{2+\alpha}\right),  \\
	\label{a0bis} \edz{\AL controllare le costanti}
	\frac{u_p(x)} {\| u_p\|_{\infty}} & \to \pm \Gamma\left(\frac{N+\a}{2+\a}\right) |x|^{-\frac{\n-2}{2}}\jcal_{\frac{\n-2}{2+\alpha}}\left(z_m\left(\frac{\n-2}{2+\alpha}\right)	|x|^{\frac{2+\alpha}{2}} \right) 
\quad  \text{ in } C^2(B_1)  .
	\intertext{Moreover denoting by $0<r_{1,p}< \dots r_{m,p}=1$ the nodal radii of $u_p$ we have}
	\label{a2} r_{i,p} & \to \left(\frac{z_i\left(\frac{\n-2}{2+\alpha}\right)}{z_m\left(\frac{\n-2}{2+\alpha}\right)}\right)^{\frac{2}{2+\a}} \qquad \text{ as } i=1,\dots m-1. &
	\end{align}
\end{theorem}

Next we exploit the characterization of the Morse index  in terms of a singular Sturm-Liouville problem given in \cite{AG-sez2} and recalled here in Subsection \ref{S:1.1}. Thanks to the convergence established in Theorem \ref{p1}, we are able to pass to the limit also in that singular problem and compute the Morse index of $u_p$ in a right neighborhood of $p=1$.
To state the related result some more notation is needed.
Since the map $\beta\mapsto z_i(\beta)$ is continuous and increasing, for every integer $m$ fixed there exist $\beta_i=\beta_i(\alpha, N)>0$  such that $z_i(\beta_i)$ (the $i^{th}$ zero of the Bessel function ${\mathcal J}_{\beta_i}$) coincides with $z_m(\frac{\n-2}{2+\alpha})$ (the $m^{th}$ zero of ${\mathcal J}_{\frac{\n-2}{2+\alpha}}$).
Next we write \[
N_j:=\begin{cases}
1 & \text{ when }j=0\\
\frac{(N+2j-2)(N+j-3)!}{(N-2)!j!} & \text{ when }j\geq 1
\end{cases}\]
for the multiplicity of the eigenvalue $\l_j:=j(N+j-2)$ of the Laplace-Beltrami operator in the sphere ${\mathbb S}_{N\!-\!1}$, and
$\lceil s\rceil=\min\{n\in \mathbb Z \, : \, n\ge s\}$ for the ceiling function.

	\begin{theorem}\label{mi-p=1} 	Let $u_p$ be a  radial solution to \eqref{H} with $m$ nodal zones.
	For every $\a\ge 0$ there is $\bar p =\bar p(\a)>1$ such that for $p\in(1,\bar p)$ the Morse index of $u_p$ is given by
	\begin{equation}\label{morsep=1}
	m(u_p)= 1+\sum_{i=1}^{m-1}\sum _{j=0}^{\left\lceil \frac{(2+\a)\beta_i -N}{2} \right\rceil} N_j  .
	\end{equation}
	if $\alpha \neq \a_{\ell,n} = (2n+N-2)/\beta_{\ell} -2$ (as $\ell=1,\dots m-1$, $n\in{\mathbb N}$).
	Otherwise if $\a= \a_{\ell,n}$   the Morse index is estimated by
	\begin{equation}\label{morsep=1estbrutta} \begin{split}	1+\sum_{i=1}^{m-1}\sum _{j=0}^{\left\lceil \frac{(2+\a)\beta_i -N}{2} \right\rceil} N_j \le  m(u_p)	\le 1 +\sum_{i=1}^{m-1}\sum _{j=0}^{\left\lceil \frac{(2+\a)\beta_i -N}{2} \right\rceil} N_j + \sum\limits_{\ell} N_{\frac{(2+\a)\beta_{\ell} -N}{2}+1 }   .	\end{split}	\end{equation}
\end{theorem}

In the particular case of positive solutions Theorem \ref{mi-p=1} recovers that $m(u_p)=1$, which is clearly true for the positive solution to the Lane-Emden equation (whose Morse index is equal to 1 for any value of the parameter $p$), and was already proved in \cite{AG14} for the H\'enon equation in dimension $N\ge 3$.
Coming to nodal solutions, the formula \eqref{morsep=1} is not totally explicit since the law $\beta\mapsto z_i(\beta)$ is not known. However the value of $z_i(\b)$ can be computed by a numerical procedure (for instance by the command \texttt{besselzero}  in MatLab), and by a dichotomy argument the approximated values of $\beta_i^m$  can be deduced.
\\
For the Lane-Emden equation ($\a=0$) the numerical approximation  suggests that  
$2(m-i)-1 < \beta_i -\frac{N}{2} < 2(m-i)$, so that 
\eqref{morsep=1} becomes
\begin{align*} 
m(u_p) & = m+\sum\limits_{i=1}^{m}(1+m-i)(N_{2i-1}+N_{2i}) ,
\intertext{ which simplifies into }
m(u_p) &= m(2m-1) 
\end{align*}
in dimension $N=2$.

In the plane the approximation procedure is elementary also for $\a > 0$, because the baseline Bessel function is $\jcal_0$, whose zeros are tabulated.
Of particular interest is the case of the radial solution with 2 nodal zones, which is  the least energy nodal radial solution (see \cite{BWW}) and  will be  denoted by $u^{\ast}_p$ in the following.
In that case  formula \eqref{morsep=1} becomes 
\begin{equation}\label{morsep=1-N=2}
m(u^{\ast}_p) = 2 \left\lceil\frac{2+\a}{2}\beta\right\rceil  , \quad \text{with $\beta \approx 2,\!305$.}
\end{equation}
	In particular the least energy nodal radial solution to the Lane-Emden equation ($\a=0$) has Morse index 6, as already noticed in \cite{GI}.
	For small positive values of $\alpha$ the Morse index remains 6, while there is a sequence of critical values $\alpha_n = 2 (n/\beta -1)$ where  the asymptotic Morse index increases.
	This phenomenon, which is enlightened here for the first time, suggests that the structure of the set of the solutions to \eqref{H} changes in correspondence of these values of $\a$, for $p$ arbitrarily close to $1$.
	
	To explore this issue further we focus onto the so called $n$-invariant solutions, introduced in  \cite{GI} in the Lane-Emden case, and studied also in \cite{AG-N=2} in the H\'enon case (for large values of $p$).
	We say that a function defined on the $2$-dimensional unitary ball $B$ is $n$-invariant if it is invariant for reflection across the horizontal axis and  for rotations of an angle $2\pi/n$, and we denote by $H^1_{0,n}$ the subspace of $H^1_{0}(B)$ made up by  $n$-invariant functions,  i.e.  in polar coordinates
	\begin{align*} 
	H^1_{0,n} &  : =    \left\{u\in    H^1_{0}(B)  \, : \, u(r,\theta) \hbox{ is even and } {2\pi}/n  \hbox{ periodic  w.r.t. } \theta, \, \hbox{ for every } r\in (0,1) \right\} .
	\end{align*}
	A by now standard compactness argument in the respective nodal Nehari manifold (see \cite{BW}) produces, for every integer $n$ and $p>1$,  a nodal solution to \eqref{H} which is $n$-invariant, and minimizes the associated energy among $n$-invariant functions. We denote by $U_{p,n}$ such {\it $n$-invariant nodal  least energy solution}.
In particular $U_{p,1}$ coincides with the {\it least energy nodal solution}, thanks to the symmetry result in \cite{BWW}, and so it is known that it is nonradial for every value of $p$.
Coming to $n\ge 2$, it is not known if $U_{p,n}$ are radial or not, neither if they are distinct one from another.
For instance in the Lane-Emden case \cite{GI} showed that for $p$ close to 1 $U_{p,n}$ are nonradial  as $n=1,2$, while they are radial and coincide with $u^\ast_p$ for $n\ge 3$.
\\
To find an answer to this questions we investigate the asymptotic profile of $U_{p,n}$  for $p$ close to 1, proving that
\begin{theorem}\label{n-asympt}
	Let $U_{p,n}$ be a $n$-invariant least energy nodal solution to \eqref{H} in dimension $N=2$.
	As $p\to 1$ we have
	\begin{align}\label{unod-p=1-fin-1}	 
	\|U_{p,n}\|_{\infty}^{p-1}\longrightarrow & 
	\left(\frac{2+\alpha}{2} \, z_{1}\left(\frac{2n}{2+\alpha}\right)\right)^2 , \\
	\label{unod-p=1-fin-2}	 
	\frac{{U_{p,n}(r,\theta)}}{\|U_{p,n}\|_{\infty}} \longrightarrow   &  
	\pm \frac{1}{\|\jcal_{\frac{2n}{2+\alpha}}\|_{\infty}} \jcal_{\frac{2n}{2+\alpha}}\left(z_{1}\left(\frac{2n}{2+\alpha}\right)r^{\frac{2+\alpha}{2}}\right)   \cos(n\theta) .
	\intertext{for $n< \frac{2+\a}{2}\beta$, and}
	\label{unod-p=1-fin-1-alti}	 
	\|U_{p,n}\|_{\infty}^{p-1}\longrightarrow & 
	\left(\frac{2+\alpha}{2} \, z_{2}(0)\right)^2 , \\
	\label{unod-p=1-fin-2-alti}	 
	\frac{{U_{p,n}(r,\theta)}}{\|U_{p,n}\|_{\infty}} \longrightarrow   &  
	\pm \jcal_{0}\left(z_{2}(0)\, r^{\frac{2+\alpha}{2}}\right)  .
	\end{align}				
	for $n>\frac{2+\a}{2}\beta$.
\end{theorem}
Here $\beta\approx 2,\!305$ is the same number as in \eqref{morsep=1-N=2}.
\\
The proof of Theorem \ref{n-asympt}, reported in  Subsection \ref{ss:sym-asympt}, is quite long and involved.
First, by a refined blow-up technique relying on the Morse index in the space $H^1_{0,n}$, we establish an estimate which ensures that $U_{p,n}$ converges to an eigenfunction of \eqref{prima-autof-weight}.
Next, taking advantage by the minimality of $U_{p,n}$, we see that its limit  must be the second eigenfunctions of  \eqref{prima-autof-weight} in the space $H^1_{0,n}$. Here is the point where the number $\frac{2+\a}{2}\beta$ comes into play, because it is the threeshold under which the second eigenfunction in $H^1_{0,n}$ is nonradial.

Starting from the asymptotic description in Theorem \ref{n-asympt} we can see that, for $p$ close to one, $U_{p,n}$ are nonradial and distinct for $n=1, \dots \left\lceil\frac{2+\a}{2}\beta -1 \right\rceil$, while they coincide with the radial nodal least energy solution $u^\ast_p$ for $n > \frac{2+\a}{2}\beta$, thus obtaining the following multiplicity result.

	\begin{theorem}\label{teo:existence-N=2} 
		In dimension $N=2$ there exists $\bar p = \bar p(\alpha)> 1$ such that  \eqref{H} has $\left\lceil\frac{2+\a}{2}\beta-1\right\rceil $ distinct nodal nonradial solutions for every $p\in (1,\bar p(\alpha))$. 	
	\end{theorem}
The solutions are distinct meaning that they cannot be obtained from each other by reflection or rotation, and of course they are  not one the opposite of the other. 
When $\a=0$  Theorem \ref{teo:existence-N=2} provides $2$ solutions and gives back the multiplicity result in \cite{GI}, from which the present one borrows many ideas.
\\
Finally we compare Theorem \ref{teo:existence-N=2} with \cite[Theorem 1.6]{AG-N=2}, concerning large values of $p$.
In that range the $n$-invariant least energy nodal solutions $U_{p,n}$ are nonradial for every $n=1, \dots \left\lceil\frac{2+\a}2 {\kappa}-1\right\rceil $, where 	$\kappa \approx 5.1869$ is another fixed number related to the Morse index of $u^{\ast}_p$ for large values of $p$. Further they are distinct one from another thanks to a monotonicity result in \cite{G-19}. 
Since $\kappa>\beta+2$, we see that  for $n=\left[\frac{2+\a}{2}\beta +1\right]  \dots \left\lceil \frac{2+\a}{2}\kappa-1\right\rceil$ the curve  $p\mapsto  U_{p,n}$ coincide with the one of radial least energy nodial solutions $p\mapsto u^\ast_{p}$ for $p$ under a certain value $p_n$, and then it bifurcates becoming nonradial.
We conjecture that, specularly, the least energy solutions $U_{p,n}$ are nonradial (respectively, radial) for every values of $p>1$
when $n \le \left\lceil\frac{2+\a}{2}\beta-1\right\rceil $ (respectively, $n\ge \left\lceil\frac{2+\a}2 {\kappa}\right\rceil$).

\section{Preliminary remarks on the limit  problem}\label{S.0}

The profile of solutions to the H\'enon problem 
\[
\left\{\begin{array}{ll}
-\Delta u = |x|^{\alpha}|u|^{p-1} u \qquad & \text{ in } B, \\
u= 0 & \text{ on } \partial B,
\end{array} \right.
\tag{\ref{H}}\]
is related to  a weighted eigenvalue problem for the Laplacian, namely
\[
\left\{
\begin{array}{ll}
-\Delta \omega= \mu |x|^{\a} \omega & \text{ in } B, \\
\omega = 0 & \text{ on } \partial B .  
\end{array}\right.
\tag{\ref{prima-autof-weight}}\]
Indeed it is easy to prove the following general fact.

\begin{lemma}\label{lem:conv-eigenv}
	Let $p_n\to 1$ and $u_n$ any nontrivial solution to \eqref{H} with $p$ replaced by $p_n$. If $\|u_n\|_{\infty}^{p_n-1}\le C$, then there exists an eigenvalue $\mu$ of \eqref{prima-autof-weight} with eigenfunction $\omega$ such that $\|\omega\|_{\infty}=1$ and  (up to an extracted sequence) 
	\begin{equation}\label{limite}
	\|u_n\|_{\infty}^{p_n-1}\longrightarrow \mu \quad  \text{ and } \quad  \frac{u_n}{\| u_n\|_{\infty}} \longrightarrow \omega  \ \text{  in $C^2(B)$ and in $C(\bar B)$.}
	\end{equation}
	\end{lemma}
\begin{proof}
Certainly $\|u_n\|_{\infty}^{p_n-1}$ converges to a nonnegative number, say it $\mu$, up to an extracted sequence.
Next $\bar{u}_n(x)=\frac {u_n(x)}{\| u_n\|_{\infty}}$ satisfies
\begin{equation}\label{nonumber}
\begin{cases}
-\Delta \bar{u}_n=|x|^{\alpha}\| u_n\|_{\infty}^{p_n-1}|\bar{u}_n|^{p_n-1}\bar{u}_n & \hbox{ in }B,\\
\bar{u}_n=0 & \hbox{ on }\partial B.
\end{cases}
\end{equation}
and is nontrivial since $\|\bar u_n\|_{\infty}=1$. Hence $\|u_n\|_{\infty}^{p_n-1}$ cannot vanish (and so $\mu> 0$) because by maximum principle
\[\bar u_n = (-\Delta)^{-1} |x|^{\alpha}\| u_n\|_{\infty}^{p_n-1}|\bar{u}_n|^{p_n-1}\bar{u}_n \le \| u_n\|_{\infty}^{p_n-1} (-\Delta)^{-1} (1) .\]

Moreover 
\begin{equation}
\left(|\bar{u}_n|^{p_n-1}-1\right)\bar{u}_n \to 0 \quad  \text{uniformly.}
\end{equation}
Indeed for any fixed $n$ we have $\left(|\bar{u}_n|^{p_n-1}-1\right)\bar{u}_n =0$ if $\bar{u}_n= 0$, otherwise from the equality
\begin{align}\label{per-dopo}
a^s-1= & s \log a \int_0^1a^{t s} dt \quad \text{ as } a>0, \ s \in \R,
\end{align}
we deduce that 
\begin{align*}
\left|\left(|\bar{u}_n|^{p_n-1} -1\right)\bar{u}_n\right|\le& (p_n-1)\left|\log|\bar{u}_n| \int_0^1|\bar{u}_n|^{1+t(p_n-1)} dt \right| \\  \le& c (p_n-1) |\bar{u}_n|^{1/2} \le   c (p_n-1) .
\end{align*}
So $\bar u_n$ converges weakly  to a function $\omega$  that solves \eqref{prima-autof-weight} for $\mu = \lim \| u_n\|_{\infty}^{p_n-1}$, and by ellipticity $\bar u_n\to \omega$ in $C^2(B)$ and uniformly on $\overline B$. From this it also follows that $\|\omega\|_{\infty}=1$, concluding the proof of \eqref{limite}. 
\end{proof}

It is not hard to obtain a better asymptotic description which shall be of use later on.

\begin{corollary}\label{cor}
Let $p_n\to 1$, $u_n$, $\mu$ and $\omega$  as in the previous Lemma, and define
\begin{equation}\label{def:c}
c=\dfrac{-\int_B|x|^{\a} \log| \omega| \, {\omega}^2 dx}{\int_B|x|^{\a} \omega^2 dx}.\end{equation}
Then as $n\to \infty$ we have
	\begin{align}\label{weak-nod}
	\|u_n\|_{\infty}^{p_n-1} & = \mu \left( 1 + c (p_n-1) \right) + o(p_n-1) , \\
	\label{point-limit}
	\mu^{-\frac{1}{p_n-1}} u_n & \longrightarrow e^{c} \omega  \quad \text{ in } C(\bar B).
	\end{align}
	\end{corollary}
This facts have been proved in \cite{GI} in some particular cases, but their proof still work in wide generality. We report it here for the reader convenience.
\begin{proof}
	To obtain \eqref{weak-nod} we compute
	\begin{align*}
	\| u_n\|_{\infty}^{p_n-1}	\int_B |x|^{\a}|\bar{u}_n|^{p_n-1}\bar{u}_n  \omega dx  & = \frac{1}{	\| u_n\|_{\infty}}\int_B |x|^{\a}|{u}_n|^{p_n-1}{u}_n \omega dx =  \frac{1}{	\| u_n\|_{\infty}}\int_B \nabla u_n \nabla \omega dx 
	\intertext{because $u_n$ solves \eqref{H}. Next using that $\omega$ solves \eqref{prima-autof-weight} we end up with}
	& =   \frac{\mu}{	\| u_n\|_{\infty}}\int_B |x|^{\a} u_n \omega dx = \mu \int_B |x|^{\a} \bar u_n \omega dx .
	\end{align*}
	Hence
	\begin{align*}
	\left(\| u_n\|_{\infty}^{p_n-1}-\mu \right) \int_B |x|^{\a} |\bar{u}_n|^{p_n-1} \bar u_n  \omega dx = \mu \int_B |x|^{\a} \left( 1 - |\bar{u}_n|^{p_n-1}\right)  \bar u_n \omega dx
\\ 
	\underset{\eqref{per-dopo}}{=}  -\mu  (p_n-1) \int_B |x|^{\a} \log|\bar u_n| \int_0^1 |\bar u_n|^{t(p_n-1)} dt  \, \bar u_n \omega dx .
	\end{align*}
	Summing up we have 
	\begin{align*}
	\frac{\| u_n\|_{\infty}^{p_n-1}-\mu}{\mu (p_n-1)} = \frac{- \int_B |x|^{\a} \log|\bar u_n| \int_0^1 |\bar u_n|^{t(p_n-1)} dt  \, \bar u_n \omega dx }{\int_B |x|^{\a} |\bar{u}_n|^{p_n-1} \bar u_n  \omega dx }  \to c
	\end{align*}
	as $n\to \infty$.
	Indeed the uniform convergence of $\bar u_n$ yields that \[\int_B |x|^{\a} |\bar{u}_n|^{p_n-1}u_n \omega dx \to \int_B |x|^{\a}{ \omega}^2 dx.\] Further 
	$\left| \log|\bar u_n| \int_0^1 |\bar u_n|^{t(p_n-1)} dt  \, \bar u_n\right| \le \sup\limits_{s\in(0,1]} |s\, \log s | < \infty$ and so the Dominated Convergence Theorem yields
	\[\int_B |x|^{\a} \log|\bar u_n| \int_0^1 |\bar u_n|^{t(p_n-1)} dt  \, \bar u_n \omega dx \to \int_B|x|^{\a} \log| \omega|\omega^2 dx.\]
	Eventually \eqref{point-limit} follows because
	\begin{align*}
	\mu^{-\frac{1}{p_n-1}} u_n = \left(\frac{\|u_n\|_{\infty}^{p_n-1}}{\mu}\right)^{\frac{1}{p_n-1}} \bar u_n =  \left( 1 + c (p_n-1) + o(p_n-1) \right)^{\frac{1}{p_n-1}} \bar u_n 
	\longrightarrow e^{c} \omega 
	\end{align*}
	 in $C(\bar B)$ by \eqref{limite}.
\end{proof}
Let us compute explicitly the eigenvalues and eigenfunctions of \eqref{prima-autof-weight}, which are related to the Bessel function  of first kind
\begin{equation}\label{bessel}
\jcal_{\beta}(r) = r^{\beta}\sum\limits_{k=0}^{+\infty} \dfrac{(-1)^ik}{k!\Gamma(k+1+\beta)} \left(\frac{r}{2}\right)^{2k},\end{equation}
when $\beta= \frac{\n-2+2n}{2+\alpha}$ for some $n\in \N$. Here and henceforth we write 
\begin{align}\label{z-1-beta}
	z_{i}(\beta) \text{  for the $i^{th}$ zero of $\jcal_{\beta}$, as $i\in\N$, $i\ge 1$,} \end{align}
$\lambda_n= n(N-2+n)$  for the sequence of the eigenvalues of the Laplace Beltrami operator on ${\mathbb S}_{N-1}$,  $N_n =\frac{(N+2j-2)(N+n-3)!}{(N-2)!n!}$ for its multiplicity, and $Y_{n,j}$ for the eigenfunctions of the Laplace Beltrami operator on ${\mathbb S}_{N-1}$, i.e. the spherical harmonics, as $n,j\in\N$, $j= 1, \dots N_n$.

\begin{lemma}\label{lem:autov-peso}
	The eigenvalues  of \eqref{prima-autof-weight} are 
	\begin{align} \label{autov-peso}
\mu_{n,i}=& \left(\frac{2+\alpha}{2} \, z_{i}\left(\frac{\n-2+2n}{2+\alpha}\right)\right)^2 
\intertext{and the related eigenfunctions are } \label{autof-peso}
\omega_{n,i}(x)= &  |x|^{-\frac{\n-2}{2}}\jcal_{\frac{\n-2+2n}{2+\alpha}}\left(z_{i}\left(\frac{\n-2+2n}{2+\alpha}\right)|x|^{\frac{2+\alpha}{2}}\right) Y_{n,j} \left(\frac{x}{|x|}\right) .
\end{align}
\end{lemma}
\begin{proof}
Let $\omega\in H^1_0(B)$ solve the equation in \eqref{prima-autof-weight}. We decompose it along the spherical harmonics and write
\begin{equation}\label{autovpeso-decomposition}
\omega(x)= \sum\limits_{n=0}^{\infty}\sum\limits_{j=1}^{N_n} \psi_{n,j}(|x|) \, Y_{n,j} \left(\frac{x}{|x|}\right) .
\end{equation}
An easy computation shows that $\psi_{n,j}$ are characterized by
\begin{equation}\label{2-app}
\begin{cases}
-\left(r^{N-1}\psi'\right)' = r^{N-1} \left( r^{\a}\mu -\frac{\lambda_n}{r^2} \right) \psi \ & \text{ for } 0<r< 1 , \\
\psi \in H^1_{0,\rad}(B)  &   \mbox{ and also } \\
{\psi}/{|x|} \in L^2(B) & \mbox{ if } n\ge 1. 
\end{cases}
\end{equation}
	Next we perform the change of variable 
	\[ t=r^{\frac{2+\alpha}{2}} , \qquad \phi(t) =  \psi(r) .\] 
The function  $\psi$ solves the equation in \eqref{2-app} if and only if 
	\begin{equation}\label{c1}
	t^2\phi^{\prime\prime}+ \frac{2N-2+\a}{2+\a} t  \phi^{\prime}+ \left(\frac{2}{2+\a}\right)^2\left(\mu  t^2-\lambda_n\right)\phi  =0  \quad  \text{ for }  0< t<1.
	\end{equation}
	If $N=2$  \eqref{c1} is a Bessel equation, otherwise is sufficient to perform a 
	further transformation, namely 
	\[ \hat\phi(t)=t^{\frac{N-2}{2+\a}} \phi(t),\] 
	to obtain the  Bessel equation
	\begin{equation}\label{c2}
	t^2{\hat \phi}^{\prime\prime}+t{\hat\phi}^{\prime}+\left(\left(\frac{2\sqrt{\mu}}{2+\a}\right)^2 t^2- \left(\frac{N-2+2n}{2+\a}\right)^2 \right) \hat\omega=0.
	\end{equation}
	Here we have also used the explicit value  $\l_n= n(N-2+n)$.
The solutions of \eqref{c2} are   linear combinations of the Bessel functions of first and second kind, precisely
\[\hat\phi(t)= C_1 \jcal_{\frac{N-2+2n}{2+\a}}\left(\frac{2\sqrt{\mu} \, t}{2+\a}\right)+ C_2{\mathcal Y}_{\frac{N-2+2n}{2+\a}}\left(\frac{2\sqrt{\mu} \, t}{2+\a}\right) . \]
Coming back to $\psi(r)= r^{-\frac{N-2}{2}} \hat\phi(r^{\frac{2+\a}{2}})$ and imposing that $\psi\in H^1_{0,\rad}(B)$ one sees  that  the  coefficient $C_2$ must be zero, and the condition $\psi(1)=0$ yields  that  $\frac{2\sqrt{\mu}}{2+\a}$ is a zero of the Bessel function $\jcal_{\frac{N-2+2n}{2+\a}}$, that is \eqref{autov-peso}.
	Eventually 
	\[ \psi(r)=  C r^{-\frac{N-2}{2}}\jcal_{\frac{N-2+2n}{2+\a}}\left( z_{n,i}  \, r^{\frac{2+\alpha}{2}}\right) ,\]
	and the decomposition \eqref{autovpeso-decomposition} yields \eqref{autof-peso}. 
\end{proof}

\begin{remark}\label{rem:autov-rad}
	The same arguments of the proof of Lemma \ref{lem:autov-peso} show that the radial eigenvalues of  \eqref{prima-autof-weight}, i.e. the eigenvalues whose corresponding eigenfunctions belong to $H^1_{\rad}(B)$, are
		\begin{align} \label{autov-rad-peso}
	\mu_{0,i}&= \left(\frac{2+\alpha}{2} \, z_{i}\!\left(\frac{N-2}{2+\a}\right)\right)^2 .
	\intertext{Each of them has only one radial eigenfunction (up to a multiplicative constant) given by } \label{autof-rad-peso}
	\omega_{0,i}(r) & =   r^{-\frac{\n-2}{2}}\jcal_{\frac{\n-2}{2+\alpha}}\left(z_{i}\!\left(\frac{N-2}{2+\a}\right)r ^{\frac{2+\alpha}{2}}\right)  .
	\end{align}
	\end{remark}

\section{Radial solutions} \label{S:1}

In  this section we deal with radial solutions,  for which a more detailed description of both the asymptotic profile and the Morse index can be given.

We consider the number of nodal zones $m$ as fixed, and write $u_p$ for the radial solution to \eqref{H} with $m$ nodal zones. It is unique up to the sign (see \cite{NN}) and to fix idea we shall take that $u_p(0)>0$. 
We denote by $0<r_{1,p}< \dots r_{m,p}=1$ the nodal radii of $u_p$, so that $u_p(r_{i,p})=0$ as $i=1,\dots m$.
It is not hard to see by ODE techniques that $u_p$ has only one critical point in any nodal interval $A_1=[0,r_{1,p})$ or $A_i=(r_{i-1,p}, r_{i,p})$ if $i=2,\dots m$, which shall be denoted by $s_{i-1,p}$ henceforth. Moreover $s_{0,p}=0$ is the global maximum point and 
$u_p(0)>-u_p(s_{1,p})>u_p(s_{2,p})>\dots (-1)^{m-1} u_p(s_{m-1,p})$. We refer to \cite[Proposition 4.1]{AG-sez2-bis} for a detailed proof.
\\
Here we see that when $p$ approaches $1$, $\|u_p\|^{p-1}$ stays bounded , none of the  nodal zones disappears 
and a suitable rescaling of $u_{p}$ converges to the $m^{th}$ radial eigenfuntion of \eqref{prima-autof-weight}, which gives Theorem \ref{p1}.  

\begin{proof}[Proof of Theorem \ref{p1}]
Let us check first that $\|u_p\|_{\infty}^{p-1}=|u_p(0)|^{p-1}$ is bounded for $p$ close to $1$.
	If not there exists a sequence $p_n\to 1$ such that
	\[ \tau_n= \|u_{p_n}\|_{\infty}^{\frac{p_n-1}{2+\alpha}}  \to \infty \quad \text{ as } \, n\to +\infty,  \]
	so we look at the rescaled function
	\[U_n(x)= \dfrac{1}{\|u_{p_n}\|_{\infty}} u_{p_n}\left(\dfrac{x}{\tau_n}\right) ,  \quad \text{ as } x\in B_{\tau_n}= \{ x \in \R^N : |x| < \tau_n\}, \]
	that satisfies 
	\begin{equation}\label{risc}
	\left\{\begin{array}{ll}
	-\Delta U_n = |x|^{\alpha} |U_{n}|^{p_n-1} U_n , & \mbox{ in } B_{\tau_n} , \\
	|U_n|\le U_n(0)= 1 ,  & \\
	U_n=0, & \mbox{ on } \partial B_{\tau_n}.
	\end{array}\right.\end{equation}
	Because $|x|^{\alpha}|U_n|$ is locally bounded, the function $U_n$ converges locally uniformly in $\R^{\n}$ to a radial function $U$ which solves 
	\[
	\left\{\begin{array}{ll}
	-\Delta U = |x|^{\alpha}  U , & \mbox{ in } \R^{\n} , \\
	|U|\le U(0)= 1 .  & 
	\end{array}\right.\]
	Remark that the number of nodal zones of $U$ can not overpass the one of $U_n$, that is $m$. Indeed inside each nodal zone $U_n$ has fixed sign and converges uniformly to $U$. Therefore $U$ cannot change sign and Hopf Lemma yields that no further zero can appear.
	\\
	On the other hand $w(r)=r^{\frac{N-2}{2+\a}}U\left(\left(\frac{2r}{2+\a}\right)^{\frac{2}{2+\a}}\right)$ solves
	a Bessel equation
	\[ r^2w'' + {r} w' + \left(r^2+\left(\frac{N-2}{2+\a}\right)^2\right) w = 0,\]
	and since it is bounded near at the origin $w(r)= A  {\mathcal J}_{\frac{N-2}{2+\a}}(r)$, where ${\mathcal J}$ stands for the Bessel function of first kind.
	This is not possible because $w$ has  an infinite number of nodal zones, proving that $\| u_p\|_{\infty}^{p-1}\le C$.
	\\
	Hence Lemma \ref{lem:conv-eigenv} ensures that when $p_n\to 1$ then	
	the function $\bar{u}_n(x)=\frac {u_n(x)}{\| u_n\|_{\infty}}$ converges to an eigenfunction $\omega$ of \eqref{prima-autof-weight} related to the eigenvalue $\mu=\lim \| u_n\|_{\infty}^{p_n-1}$. Of course $\omega$ has to be radial, it remains to show that it has exactly $m$ nodal zones.
	Actually the first nodal zone, say it $B_{r_n}=\{ x \, : \, |x|<r_n\}$, can not collapse to a null set because multiplying the equation in \eqref{H} by $u_n$ and integrating on $B_{r_n}$ one sees that
	\begin{align*}
	{\int_{B_{r_n}} |\nabla  u_n|^2 dx} & = {\int_{B_{r_n}}|x|^{\alpha}  | u_n|^{p+1} dx}  \le {r_n^{\alpha} \|u_n\|_{\infty}^{p_n-1} \int_{B_{r_n}}  |u_n|^2 dx} \\
	& \le C r_n^{N+\alpha} \|u_n\|_{\infty}^{p_n-1} {\int_{B_{r_n}} |\nabla  u_n|^2 dx}
	\end{align*}
	by the Poincar\'e inequality in the ball $B_{r_n}$. Remark that it also follows that $\|u_n\|_{\infty}^{p_n-1}$ does not vanish and in particular $\mu>0$.
	\\
	The last nodal zone  can not disappear either. To see this fact we denote by $s_n$ the last zero of $u_n$, $\tau_n=\|u_n\|_{\infty}^{\frac{p_n-1}{2+\alpha}}$ as before and rule out  the occurrence $R_n:=\tau_n(1-s_n)\to 0$. 
		To this aim we look at the rescaled sequence
	\[ \zeta_n(r) = \dfrac{1}{\|u_n\|_{\infty}} \left| u_n\left(s_n+ \dfrac{r}{\tau_n}\right)\right|,  \quad \text{ as } 0< r < R_n \]
	Now $0<\zeta_n\le 1$ on $(0, R_n)$ and it  solves
	\[\begin{cases}
	-\left( (r+\tau_ns_n)^{N-1}\zeta_n'\right)'  = (r+\tau_ns_n)^{N-1+\alpha} \zeta_n^p  & 0<r<R_n, \\
	\zeta(0)=\zeta(R_n) = 0. &
	\end{cases}\]
	So, recalling that we have already proved that $\tau_n s_n\ge \delta >0$, we compute
	\begin{align*}
	\int_0^{R_n} |\zeta_n'|^2 dr &\le \frac{1}{\delta^{N-1}}\int_0^{R_n} (r+\tau_ns_n)^{N-1}|\zeta_n'|^2 dr = \frac{1}{\delta^{N-1}} \int_0^{R_n} (r+\tau_ns_n)^{N-1+\a}\zeta_n^{p_n+1} dr  \\
	&\le \frac{(R_n+s_n\tau_n)^{N-1+\a} \|\zeta_n\|_{\infty}^{p_n-1}}{\delta^{N-1}} \int_0^{R_n} \zeta_n^{2} dr
	\intertext{and using Wirtinger inequality and the boundedness of $R_n+s_n\tau_n$ and $\|\zeta_n\|_{\infty}$ we end up with}
	&\le C R_n^2 \int_0^{R_n} |\zeta_n'|^2 dr ,
	\end{align*}
	which forbids $R_n\to 0$. Similarly one can see that none of the other nodal zones can vanish, and so 
	 $\omega$ has at least $m$ nodal zones. But no more nodal zones can appear because inside each nodal zone $\omega$ is the uniform limit of  $\bar u_n$ which has fixed sign.	
	\\
	Eventually $\mu$ is the $m^{th}$ radial eigenvalue for \eqref{prima-autof-weight} and Remark \ref{rem:autov-rad} completes the proof of \eqref{a0}, \eqref{a0bis}, \eqref{a2}.
	The constant in \eqref{a0bis} comes from the condition $\omega (0)=1$.
\end{proof}

\subsection{Computation of the Morse index}\label{S:1.1}

In this subsection we prove Theorem \ref{mi-p=1}, by taking advantage of a characterization of the Morse index given in \cite{AG-sez2}.
Let us recall that the Morse index of any solution $u_p$ to \eqref{H}, that we denote by $m(u_p)$, is connected with the linearized operator 
\begin{equation}\label{Lp}	L_{u_p} \psi:=-\Delta \psi-p |x|^\a|u_p|^{p-1}\psi , \end{equation}
and the quadratic form
\begin{equation}\label{Qp}	{\mathcal Q}_{u_p} \psi = \int _B |\nabla \psi|^2-p\int_B|x|^\a|u_p|^{p-1}\psi^2, \end{equation}
and can be defined as the number, counted with multiplicity, of the negative eigenvalues of  
\begin{equation}\label{eigenvalue-problem}
\left\{\begin{array}{ll}
L_{u_p} \psi=\L_k(p)\, \psi & \text{ in } B\\
\psi\in H^1_0(B), & 
\end{array} \right.
\end{equation}
or equivalently as the maximal dimension of a subspace of $H^1_0(B)$ where ${\mathcal Q}_{u_p} $ is negative defined.
In \cite[Proposition 1.1]{AG-sez2} an alternative definition of Morse index has been given by using a singular eigenvalue problem
	\begin{equation}\label{sing-eigenvalue-problem}
	\left\{\begin{array}{ll}
	L_{u_p} \widehat \psi=\dfrac{\widehat\L_k(p)}{|x|^2} \widehat \psi & \text{ in } B\\
	\widehat \psi\in \mathcal H_0 ,& 
	\end{array} \right.
	\end{equation}
	where ${\mathcal H}_{0}$ denotes the subspace of $H^1_{0}(B)$ 
	\[ {\mathcal H}_{0}=\big\{ \psi  \in H^1_{0}(B) \, : \, |x|^{-1}\psi \in L^2(B)\big\} .\]
Precisely $m(u_p)$ is the number, counted with multiplicity, of the negative eigenvalues of  \eqref{sing-eigenvalue-problem},  or equivalently as the maximal dimension of a subspace of $\mathcal H_0$ where ${\mathcal Q}_{u_p} $ is negative defined.
\\
Hereafter we focus on the radial solution with exactly $m$ nodal zones, which shall be denoted by $u_p$, again. For such solution an even more effective description of its  Morse index  can be done by taking advantage from  the  transformation introduced in \cite{GGN}
\begin{equation}\label{transformation-henon}
t=r^{\frac{2+\a}{2}} ,\qquad w(t)=u(r) ,
\end{equation} 
which maps the space $H^1_{0,\rad}(B)$ into
\begin{equation}\label{H0M-def}
\begin{array}{rl}H^1_{0,M}:= \big\{v:(0,1)\to\R\, : \, & v \text{ is measurable and  has a first order weak derivative $v'$} \\
& \text{ with } \int_0^1 r^{M-1} \left(v^2 + |v'|^2 \right) dr < +\infty \big \}, \end{array}
\end{equation}
for 
\begin{align}\label{Malpha}
M & = M(N,\alpha):= \frac{2(N+\alpha)}{2+\alpha}  .
\end{align} 
As shown in \cite[Proposition 4.9]{AG-sez2}, $u_p$ 
is transformed by \eqref{transformation-henon} into the unique (up to the sign) solution of the "radially extended" Lane-Emden problem
\begin{equation}\label{LE-radial}
\begin{cases}
- \left(t^{M-1} w^{\prime}\right)^{\prime}= \left(\frac{2}{2+\a}\right)^2 t^{M-1} |w|^{p-1}w  , \qquad  & 0<t< 1, \\
w'(0)=0, \quad w(1)=0
\end{cases}\end{equation}
which has $m$ nodal zones  and shall be denoted by $w_p$ hereafter. 
\\
In the same paper the computation of the Morse index of $u_p$ has been related to a singular  Sturm-Liouville problem connected with $w_p$, namely 
\begin{equation}\label{radial-singular-problem-LE}
\begin{cases}-\left(t^{M-1}\phi_i'\right)'=t^{M-1}\left( W_p (t)+\frac{ \nu_i(p)}{t^2}\right) \phi_i  &  \text{ for } t\in(0,1) ,
\\
\phi_i \in {\mathcal H}_{0,M} & 
\end{cases}
\end{equation}
where
\begin{equation} \label{pot}
W_p(t)= p \left(\frac{2}{2+\a}\right)^2 |w_p(t)|^{p-1}
\end{equation}
and ${\mathcal H}_{0,M}$ denotes the subspace of $H_{0,M}$ made up by functions which also satisfy
\begin{equation}\label{Hstorto0M}
\int_0^1 t^{M-3} \phi^2 dt < \infty.
\end{equation}
It is useful to remark that the eigenvalues $\nu_i(p)$ are well defined only if $\nu_i(p)< \left(\frac{N-2}{2+\a}\right)^2$, and in this case they also have a variational characterization
\begin{equation}\label{radial-singular-M}\begin{split}
{\nu}_1(p )= & \inf_{\substack{\phi\in\mathcal{H}_{0,M}\ w\neq 0}}\frac{\int_0^1 t^{M-1}\left(|\phi'|^2- W_p\phi^2\right) dt}{\int_0^1 t^{M-3}\phi^2 dt},  
\\
{\nu}_{i}(p)= & \inf_{\substack{\phi\in\mathcal{H}_{0,M}\ \phi\neq 0\\ w\underline \perp_{M}\{\phi_1,\dots,\phi_{i-1}\}}}\frac{\int_0^1 t^{M-1}\left(|\phi'|^2- W_p\phi^2\right) dt}{\int_0^1 t^{M-3}\phi^2 dt} , 
\end{split}\end{equation}
see as shown in \cite[Subsection 3.1]{AG-sez2}. Here the perpendicularity condition denoted by $\underline{\perp}_M$ means 
\begin{equation}\label{perpM-def}
\phi \underline{\perp}_M \psi \ \iff \ \int_0^1 t^{M-3} \phi \psi dt =0 . 
\end{equation}
Moreover by the analysis performed in \cite[Section 4]{AG-sez2-bis} the only negative eigenvalues of \eqref{radial-singular-problem-LE} are $\nu_1(p)<\nu_2(p)<\dots< \nu_m(p)<0$ and they satisfy
\begin{align}
\label{nl<k-general-H} & {\nu}_i(p)  < -(M-1)   &\text{ as } i=1,\dots m-1 ,
\\
\label{num>k-general-H} & -(M-1) <{\nu}_m(p) <0  ,  &
\end{align}
for any value of the parameter $p$.

Putting together \cite[Proposition 1.5]{AG-sez2} and \cite[Theorem 1.3]{AG-sez2-bis} gives
\begin{proposition}\label{prop:morse-formula}
	Let $u_p$ be a radial solution to \eqref{H} with $m$ nodal zones. Then it is radially nondegenerate, its radial Morse index is $m$ and its Morse index is given by
\begin{equation}\label{morse-formula}
m(u_p) = \sum\limits_{i=1}^{m}\sum\limits_{j=0}^{\lceil J_i(p)-1\rceil } N_j, 
\end{equation}
\begin{tabular}{ll} 
	where & $J_i(p)=\frac{2+\a}{2} \left(\sqrt{\left(\frac{N-2}{2+\a}\right)^2- \nu_i(p)}-\frac{N-2}{2+\a}\right)$,\\ 
	& $\lceil s \rceil = \{\min n\in \mathbb Z \, : \, n\ge s\}$ denotes the ceiling function and \\
	& $	N_j=\begin{cases}
	1 & \text{ when }j=0\\
	\frac{(N+2j-2)(N+j-3)!}{(N-2)!j!} & \text{ when }j\geq 1
	\end{cases}$ is the multiplicity of the eigenvalue  \\
	&
	$\l_j=j(N+j-2)$ for the Laplace-Beltrami operator in the sphere ${\mathbb S}_N$.
\end{tabular}

Furthermore the negative singular eigenvalues  of \eqref{sing-eigenvalue-problem} can be decomposed as
\[\widehat\L_k(p) = \left(\frac{2+\a}{2}\right)^2 \nu_i(p) + \l_j \]
 as far as $ \nu_i(p)<- \left(\frac2{2+\a}\right)^2  \l_j$ for some $i=1, \dots m$ and $j\ge 0$, while the related eigenfunctions are 
\begin{equation}\label{decomposition} \widehat \psi_k(x) = \phi_i\left(|x|^{\frac{2+\a}{2}}\right) Y_j\left(\frac{x}{|x|}\right),
\end{equation}
where $\phi_i$ is an eigenfunction of \eqref{radial-singular-problem-LE} related to $\nu_i(p)$, and $Y_j$ is an eigenfunction for the Laplace-Beltrami operator in ${\mathbb S}_{N-1}$ related to $\l_j$.
\end{proposition}
In this way the asymptotic Morse index of $u_p$ can be computed by investigating the eigenvalues $\nu_i(p)$ of \eqref{radial-singular-problem-LE}, which shall be the topic of the remaining of this subsection.

\

As a preliminary it is worth noticing that the convergence stated by Theorem \ref{p1} translates into the following one for $w_p$ and $W_p$, defined respectively in \eqref{transformation-henon} and \eqref{pot}.
\begin{corollary}\label{p1v}
	Let $z_m$ stand for $z_m\left(\frac{N-2}{2+\a}\right)$ according to \eqref{z-1-beta}.
As $p\to 1$ we have  
\begin{align}
\label{a0v}	\| W_p\|_{\infty} & \to z_m^2  	, \\
	\label{a0bisv}
	\frac{w_p(t)} {\| w_p\|_{\infty}} & \to \Gamma\left(\frac{N+\a}{2+\a}\right)  t^{-\frac{N-2}{2+\a}} \jcal_{\frac{N-2}{2+\a}} (z_m t)  \\ \nonumber
	& \quad = \Gamma\left(\frac{N+\a}{2+\a}\right) \sum\limits_{k=0}^{+\infty} \dfrac{(-\frac{z_m^2}{4})^k  }{k!\Gamma(k+\frac{\n+\a}{2+\alpha})} t^{2 k}	\qquad  \text{ in } C^2[0,1)  , 
	\intertext{ and denoting by $0<t_1<\dots t_m=1$ the zeros of $w_p$}
	\label{a2v}
	 t_i & \to \frac{z_i}{z_m}  \qquad \text{ as } i=1,\dots m-1.
	 \intertext{Besides}
	\label{a1v} W_p & \rightharpoonup z_m^2 \quad \text{ weakly in $L^2(0,1)$.} \end{align}
\end{corollary}
\begin{proof}
\eqref{a0v},  \eqref{a0bisv} and \eqref{a2v} follow immediately by \eqref{a0}, \eqref{a0bis} and \eqref{a2}.
Further $W_p(t)= \| W_p\|_{\infty}  \left(\frac{w_p(t)} {\| w_p\|_{\infty}} \right)^{p-1}$ is bounded and converges uniformly to the constant $z_m^2$ on any closed interval contained in $[0,1)\setminus\{z_1/z_m,\dots z_{m-1}/z_m\}$, so that the weak convergence follows trivially.
\end{proof}

Also in the following when  the parameter $\beta$ is omitted we mean $z_i= z_i\left(\frac{N-2}{2+\a}\right)$.
Since the map $\beta\mapsto z_i(\beta)$ is continuous and increasing (see for instance \cite{Elb}), there exists 
\begin{equation}\label{beta-i-def} \begin{split}
\beta_i=\beta_i(\alpha, N,m)>0 \ & \text{  such that }  z_i(\beta_i) \text{ (the $i^{th}$ zero of the Bessel function ${\mathcal J}_{\beta_i}$)} \\
& \text{ coincides with $z_m$ (the $m^{th}$ zero of ${\mathcal J}_{\frac{\n-2}{2+\alpha}}$)}, \end{split}\end{equation}
moreover  
\[ \beta_1>\dots\beta_m=\frac{\n-2}{2+\alpha} .\]

The limit of the singular eigenvalues $\nu_i(p)$ can be expressed in terms of the parameters $\beta_i$ as follows.

\begin{proposition}\label{formulainutile}
	Consider the eigenvalue problem  \eqref{radial-singular-problem-LE}, with $W_p$ given as in \eqref{pot}, and $u_p$ is the radial solution to \eqref{H} with $m$ nodal zones.
	Then as $p\to 1$ we have 
	\begin{align}\label{nup=1}
	 \nu_i(p)& \to \left(\frac{\n-2}{2+\alpha}\right)^2 -\beta_i^2 \qquad \text{ as } i=1,\dots m.
	 	\end{align}
	\end{proposition}

In particular $\nu_m(p)\to 0$.
Putting together Propositions \ref{prop:morse-formula} and \ref{formulainutile} yields Theorem \ref{mi-p=1}.

\begin{proof}[Proof of Theorem \ref{mi-p=1}]
		From the limit  \ref{nup=1} one sees that the index $J_i(p)$ appearing in the Morse index formula \eqref{morse-formula} satisfies
		\begin{equation}	\label{Jp=1}   
		J_i(p) \to  \frac{(2+\a)\beta_i -(N-2)}{2}  
		\end{equation}
		as $p\to 1$.
		So when  $\frac{(2+\a)\beta_i -(N-2)}{2} $ are not  integer \eqref{morsep=1} follows, while when $\frac{(2+\a)\beta_i -(N-2)}{2} $ is integer for some $i$ we only get \ref{morsep=1estbrutta}.
		\end{proof}

Some preliminary lemmas are useful to prove Proposition \ref{formulainutile}.
First we remark that all the eigenvalues of \eqref{radial-singular-problem-LE} are bounded from below in a neighborhood of $p=1$. 

\begin{lemma}\label{lem:nu-bounded}
	There exists $C>0$ such that $\nu_1(p)\ge -C$ for $p$ close to $1$. 
\end{lemma}
\begin{proof}
	By \eqref{a0v} $0\le W_p(t)\le z_m^2+\e$ for $p$ sufficiently close to $1$. So for all $\psi\in{\mathcal H}_{0,M}$
\begin{align*} 
\int_0^1 t^{M-1} \left(|\psi' |^2- W_p\psi^2\right)dt \geq -(z_m^2+\e)\int_0^1 t^{M-1}\psi^2 dt  \geq - (z_m^2+\e)   \int_0^1 t^{M-2} \psi^2 dt ,
\end{align*}
and the claim follows by the variational characterization \eqref{radial-singular-M}.  
\end{proof}

Next we establish  an ad-hoc Poincar\'e inequality. For $0\le a<b$ we denote by $H_{0,M}(a,b)$ the space of functions of $H^1(a,b)$ such that $\psi(b)=0$, endowed with the  norm
\[\|\psi; H_{0,M}(a,b) \| = \int_a^b t^{M-1} |\psi' |^2 dt .\]
It is clear that $H_{0,M}(0,1)$ is the space $H_{0,M}$ already introduced.
 It is very easy to see that
 \begin{lemma}\label{poincare}
 	 For every $\psi\in H_{0,M}(a,b)$ we have
 	\[\int_a^b t^{M-1} \psi^2 dt  \le \frac{b(b-a)}{M-1} \int_a^b t^{M-1} |\psi'|^2 dt . \]
 \end{lemma}
\begin{proof}
	Since $\psi$ has first derivative in $L^2$, it is continuous and differentiable a.e., and from $\psi(b)=0$ we get
	\[ \psi(t) = \int_t^b \psi'(r) dr .\]
	Hence
	\begin{align*}
	\int_a^b t^{M-1} \psi^2 dt & = 	\int_a^b t^{M-1} \left(\int_t^b \psi'(s) ds\right)^2 dt \\
	& \underset{\text{Holder}}{\le} \int_a^b t^{M-1}  (b-t) \int_t^b |\psi'(s)|^2 ds \,dt \le b(b-a)\int_a^b t^{M-2} \int_t^b |\psi'(s)|^2 ds \,dt \\
	&= b(b-a)\int_a^b  |\psi'(s)|^2  \int_a^s t^{M-2} dt \,ds \le \frac{b(b-a)}{M-1} \int_0^1 s^{M-1}  |\psi'(s)|^2  ds.
	\end{align*}
\end{proof}

\begin{proof}[Proof of Proposition \ref{formulainutile}]
By \eqref{nl<k-general-H}, \eqref{num>k-general-H} and Lemma \ref{lem:nu-bounded}  for any sequence $p_n\to 1$ there is an extracted sequence (that we still denote by $p_n$) such that $\nu_i(p_n)$ converges to some  $\bar{\nu}_i$. Moreover $\bar\nu_i \le - (M-1)$ if $i=1,\dots m-1$ and  $- (M-1)\le \bar \nu_m\le 0$. 
\\
Let $\psi_{i,n}\in {\mathcal H}_{0,M}$ the eigenfunction related to $\nu_i(p_n)$ normalized so that $\|\psi_{i,n}\|_{\infty}=1$. We recall that by \cite[Proposition 3.8 and Property 5 in Subsection 3.1]{AG-sez2} $\psi_{i,n} \in C[0,1]\cap C^1(0,1]$ has exactly $i$ nodal zones and for $t$ next to $0$ and
	\begin{equation}\label{psi-in-0}
	|\psi_{i,n}(t)| \le C t^{\theta_{i,n}} ,  \quad |\psi_{i,n}'(t)| \le C t^{\theta_{i,n}-1} , 
	\end{equation}
	with $\theta_{i,n}= \sqrt{\left(\frac{M-2}{2}\right)^2 -\nu_i(p_n)} - \frac{M-2}{2}$. It is worth remarking that the constants $C$ appearing here only depend by $\|\psi_{i,n}\|_{\infty}=1$, and therefore are general in the present situation.
	Further for $i=1,\dots m-1$ the  estimates in \eqref{psi-in-0} assures that $\psi_{i,n}$ are equicontinuous on a set of type $[0,\e]$ because $\theta_{i,n} \ge 1$.
	This is not the case for $i=m$.
\\	
So the proof of \eqref{nup=1} in the case $i=m$ will differ from the one for $i\le m-1$.

\

{\it First step: the first $m-1$ eigenvalues.}
We show that \eqref{nup=1} holds as $i=1,\dots m-1$, and in doing so we also see that
\begin{equation}\label{limit-autofunz}
\psi_{i,n} (t) \to A_i   t^{-\frac{N-2}{2+\a}} \jcal_{\beta_i} (z_m t) \quad \text{ uniformly in } [0,1]
\end{equation}
for some constant $A_i\neq 0$.

Using $\psi_{i,n}$ as a test function in \eqref{radial-singular-problem-LE} gives
\begin{align} \nonumber
	\int_0^1t^{M-1} |\psi_{i,p_n}'|^2 dt & =
\int_0^1t^{M-1} \left( W_p + \frac{\nu_{i}(p_n)}{t^2}\right)\psi_{i,p_n}^2 dt  \\ \label{psi'-unif} 
&	\underset{\eqref{nl<k-general-H}}{<}
	\int_0^1t^{M-1} W_p \psi_{i,p_n}^2 dt 
	\le C
	\end{align}
	thanks to the normalization of $\psi_{i,n}$ and  \eqref{a0v}.
	Hence by the compact embedding of $H^1_{0,M}$ (see \cite[Lemma 6.4]{AG-sez2})
	$\psi_{i,n}$ converges to a function $\psi_i$ weakly in $H^1_{0,M}$, strongly in any $L^q_M$ for any $q>1$ (if $M=2=N$) or for any $1<q<\frac{2M}{M-2}=\frac{2(N+\a)}{N-2}$ (if $M>2$, i.e. $N\ge 3$), and pointwise a.e. 
	Moreover $\psi_{i,n}\to \psi_i$ also in  uniformly on $[0,1]$ by Ascoli Theorem. Indeed we have already noticed that $\psi_{i,n}$ are equicontinuous on $[0,\e]$, while while for $t_1, t_2\in[\e,1]$ we have 
	\begin{align*}
	|\psi_{i,n}(t_1) - \psi_{i,n}(t_2)| \le \int_{t_1}^{t_2} |\psi'_{i,n}(s)| ds \underset{\stackrel{\text{Holder}}{\text{and } \eqref{psi'-unif}}}{\le} 
			C \left(\int_{t_1}^{t_2}s^{1-M}\right)^{\frac{1}{2}}
	\le C \e^{1-M} \sqrt{|t_1-t_2|} .
	\end{align*}
	Thanks to this and to the weak convergence in \eqref{a1v} one can pass to the limit into equation \eqref{radial-singular-problem-LE} and see that $\psi_i$ is a weak solution to
	\begin{equation}\label{limit-eigenvalue-problem}
	-\left(t^{M-1}\psi_i'\right)' = t^{M-1}\left( z_m^2 +\frac{\bar\nu_i}{t^2} \right)\psi_i  \quad  \text{ as } 0<t < 1.
	\end{equation}
	The uniform convergence yields also that $\psi_i$ is not trivial (actually $\|\psi_i\|_{\infty}=1$ by the normalization) and has at most $i$ nodal zones.
	Let us check that it has exactly $i$ nodal zones, i.e. that none of the nodal zones of $\psi_{i,n}$ disappear.
	Let $a_n$, $b_n$ be two consecutive zeros of $\psi_{i,n}$, now the function  $\psi_{i,n}$ restricted to the $(a_n,b_n)$ belongs to the space $H_{0,M}(a_n,b_n)$ introduced before Lemma \ref{poincare}, and clearly extending it to zero outside $(a_n,b_n)$ gives a function of $H_{0,M}$. Using this extension as a test function in  \eqref{radial-singular-problem-LE}  one sees that
	\begin{align*}
	\int_{a_n}^{b_n} t^{M-1} (\psi'_{i,n})^2 dt & = \int_{a_n}^{b_n} t^{M-1}\left( W_p + \frac{\nu_i(p_n)}{t^2}\right) \psi_{i,n}^2 dt
	\le  \int_{a_n}^{b_n} t^{M-1} W_p \psi_{i,n}^2 dt \\&
	 \underset{\eqref{a0v}}{\le} 
	C  \int_{a_n}^{b_n} t^{M-1}\psi_{i,n}^2 dt \le \frac{C b_n(b_n-a_n )}{M-1}  \int_{a_n}^{b_n} t^{M-1} (\psi'_{i,n})^2 dt 
	\end{align*}
	by the Poincar\'e inequality established in Lemma \ref{poincare}. 
	If follows at once that neither $b_n$ or $b_n-a_n$ vanishes.
	
	Next we define $\phi_i(t) = t^{\frac{N-2}{2+\a}} \psi_i (t/z_m)$. Starting from \eqref{limit-eigenvalue-problem} it is easily seen that $\phi_i$  solves the Bessel equation
	\begin{equation}\label{limit-bessel}
t^2 \phi_i'' + t \phi_i' + (t^2 - \b_i^2) \phi_i= 0    \qquad \text{ for } \beta_i^2= z_m^2-\bar \nu_i.
	\end{equation}
	Since $\phi_i(0)= 0$ we have that $ \phi(t)= C \jcal_{\beta_i}(t)$,
	and as $\phi_i(z_m)= z_m^{\frac{N-2}{2+\a}} \psi_i(1)=0$ it follows that $z_m$ has to be a zero of the Bessel function $\beta_i$.
	Moreover we have seen that $\psi_i$ (and then also $\phi_i$) has $i$ nodal zones, so that $\beta_i$ is determined by the condition \eqref{beta-i-def}, which implies at once \eqref{nup=1} and \eqref{limit-autofunz}.

\

{\it Second step: the last negative  eigenvalue.} It remains to check \eqref{nup=1} for $i=m$, i.e. $\bar\nu_m=0$. To do this  we compare the eigenfunction $\psi_{m,n}$  with $\bar w_n(t)= \frac{1}{\|{w}_{p_n}\|_{\infty}} {w}_{p_n}(t)$, that satisfies
\[ \begin{cases} -\left(t^{M-1} \bar w_n'\right)'= \frac{1}{p_n} t^{M-1} W_{p_n} \bar w_n  & \text{ as } 0 < t < 1 , \\
\bar w_n(0)=1 , \  \bar w_n(1)=0 . & \end{cases} \]
It is easy to establish the Picone type identity	
\begin{equation}\label{picone-m}  
\left(t^{M-1}(\psi_{m,n}'\bar w_n - \psi_{m,n}\bar w_n')\right)' =  t^{M-1}  \left(\left(\frac{1}{p_n}-1\right) W_{p_n}(t) - \frac{\nu_m(p_n)}{t^2}\right) \psi_{m,n} \bar w_n 
\end{equation}
as $0<t<1$. For a rigorous computation without requiring that $\psi_{m,n}$ is a classical solution we refer to \cite[Lemma 3.13]{AG-sez2}.
Thanks to \eqref{a0v} $\left(\frac{1}{p_n} -1\right) W_{p_n} \to 0$ uniformly, and assuming by contradiction that $\nu_m(p_n)\to \bar \nu < 0$ it follows that
\begin{equation}\label{hp-SP}
\left(\frac{1}{p_n} -1\right)W_{p_n} -\frac{\nu_m(p_n)}{t^2}  >  0 \quad \text{ as } 0<t<1 
\end{equation}
for large $n$.
Since both $\psi_{m,n}$ and $\bar w_n$ are null in $t=1$ and have exactly $m-1$ zeros on $(0,1)$, 
the Sturm-Picone's comparison Theorem yields that or $\psi_{m,n}$ is proportional to $ \bar w_n$, or all the zeros of $ \bar w_n$ follows the first zero of $\psi_{m,n}$, say it $\rho$.
The first event is not possible because $\psi_{m,n}(0)=0$ by \eqref{psi-in-0} while  $ \bar w_n(0)=1$.
In the second case we may assume w.l.g.~that $ \bar w_n, \psi_{m,n}>0$ on $(0,\rho)$, so that $\psi_{m,n}'(\rho)<0$ and $ \bar w_n(\rho)>0$. Next integrating \eqref{picone-m} between $t$ and $\rho$ and then letting $t\to 0$ gives
\begin{align*}	
\rho^{M-1}\psi_{m,n}'(\rho)\bar w_n(\rho) - \lim\limits_{t\to 0}t^{M-1}(\psi_{m,n}'\bar w_n - \psi_{m,n}\bar w_n')
\\ =
\int_0^{\rho} t^{M-1}  \left(\left(\frac{1}{p_n}-1\right) W_{p_n} - \frac{\nu_m(p_n)}{t^2}\right) \psi_{m,n} \bar w_n  dt \underset{\eqref{hp-SP}}{>} 0.
\end{align*}
But \eqref{psi-in-0} guarantees that $\psi_{m,n}(t) , t^{M-1} \psi_{m,n}'(t) \to 0$ as $t\to 0$. Indeed also when $M=2$ we have $t^{M-1} |\psi_{m,n}'(t)| \le C t^{\sqrt{-\nu_m(p_n)}}$ with $\nu_m(p_n)<0$.
So we have reached the contradiction $\psi_{m,n}'(\rho)\bar w_n(\rho) > 0$, which concludes the proof.
\end{proof}

Though the characterization of the Morse index in terms of the zeros of the Bessel function presented in \eqref{morsep=1} could not be completely satisfactory, because the laws $\beta\mapsto z_i(\beta)$ are not known explicitely,  the position of $z_i(\beta)$ can be approximated by a numerical procedure, for instance  by the command \texttt{besselzero}  in MatLab. Combining this approximation with a dichotomy argument provides the approximated values of $\beta_i$ and therefore the Morse index of $u_p$.
It is worth remarking that the computation of the asymptotic Morse index can be made more explicit in some particular cases.

	\begin{remark}[the planar case]\label{rem:N=2}
	In the plane the baseline Bessel function is $\jcal_0$, whose zeros are tabulated. Therefore approximated values of the parameters $\beta_i$ can be obtained in an elementary way.
	For instance in the case of two nodal zones we get $\beta_1\approx 2,\!305$, 
	and formula \eqref{morsep=1} yields that for $p$ near at $1$
		\begin{align*}\tag{\ref{morsep=1-N=2}}
	m(u_p) = 2 \left\lceil\frac{2+\a}{2}\beta\right\rceil  ,   \qquad  & \text{for } \beta \approx 2,\!305, 
	\intertext{ if $\a\neq\a_n=2(n/\beta -1)$ (as $n\ge 3$), otherwise }
	\nonumber 
	(2+\a_n)\beta  \le m(u_p) \le  (2+\a_n)\beta + 2 . & 
	\end{align*}
	In particular the solution to the Lane-Emden equation ($\a=0$) with two Nodal zones has Morse index 6, as already noticed in \cite{GI}.
	For small positive values of $\alpha$ the Morse index remains 6, while there is a critical value $\approx 0,6030$ above which the asymptotic Morse index increases to 8.
	This fact suggests that the set of the solutions to \eqref{H} changes in correspondence of that value of $\a$. We shall come back on this topic in next section.
	\end{remark}

In higher dimension the approximation of the parameters $\beta_i$ appearing in the computation of the Morse index can be numerically performed after having chosen a specific value for $\alpha$, which fixes the baseline Bessel function ${\mathcal J}_{\frac{N-2}{2+\a}}$. 
There is numerical evidence that \[ z_{i}(\beta+2(m-i))<z_{m}(\beta)<z_{i}(\beta+2(m-i)+1).\] 
Such estimate is not sufficient to single the  exact  Morse index out, except that in the Lane Emden case, for which one can infer that $\lim\limits_{p\to 1}J^m_i(p) \in \left( 2(m-i), 2(m-i)+1\right)$ and then, eventually
	\begin{align} \label{morsep=1le}
m(u_p) = m+\sum\limits_{i=1}^{m-1}(m-i)(N_{2i-1}+N_{2i}) \quad \text{ for $p$ close to $1$.}
\end{align}
Let us remark that in dimension  $2$ \eqref{morsep=1le} generalizes the computation in \cite{GI} (concerning the case of two nodal zones) as
\[ 	m(u_p) = m(2m-1) .	\]

\section{$n$-invariant solutions}\label{sec:sym-sol}

 We focus here on  dimension $N=2$ with the aim of producing nonradial solutions for the almost-linear Henon problem (i.e. for $p$ close to 1). This  is possible only in the framework of nodal solutions by the uniqueness result in \cite[Theorem 3.1]{AG14}, which can be easily extended to the planar case.
 Let us remark by now that the H\'enon problem is invariant for rotations around the origin, therefore any nonradial solution $u_p$ generates a family of nonradial solutions (the ones obtained by rotating $u_p$ of any given angle), and any  claim about nonradial solutions can be stated \textquotedblleft up to rotation\textquotedblright.
 
 We denote by $\mathcal E_p$ the energy functional associated to \eqref{H} i.e.
 \begin{align*}
 \mathcal E_p(u) &: =\frac 12 \int _B |\nabla u|^2-\frac 1{p+1}\int_B |x|^\a |u|^{p+1},
 \intertext{ by $\mathcal E'_p(u)$  for its Fr\'echet derivative computed at $u$, i.e.}
\mathcal E'_p (u) \cdot v & =  \int _B \nabla u \nabla v \, dx -\int_B |x|^\a |u|^{p-1} u v \, dx,
\intertext{and we introduce the Nehari set and the nodal Nehari set as}
 \mathcal N_p& :=\{v\in H^1_0(B) \, : \, v\neq 0, \ \mathcal E'_p(v)\cdot v=0 \} , 
 \\
  \mathcal N^{\nod}_p& :=\{ v\in H^1_0 \, : \,  v^{\pm}\neq 0, \ \mathcal E'_p(v)\cdot v^{\pm}=0  \} .
 \end{align*}
  Here $v^+$ and $v^-$ stand for  the positive and the negative part of $v$ respectively.
 \\
 Due to the compact embedding of $H^1_0(B)$  in $L^p(B)$ for every $p>1$,  $\min_{u\in \mathcal N^{\nod}_{p}}\mathcal E_ p(u)$ is attained at a nontrivial function, which is a weak and also classical solution to \eqref{H}, and is known as the {\it least energy nodal solution} since it changes sign  by construction.
 The existence of such least energy nodal solutions, together with some general properties,  have been established in  \cite{BW} and \cite{BWW}.
 Let us recall the ones which shall turn useful to the present purpose.
 
 \begin{proposition}\label{comeBWW}
 	Let $U_p$ be a  least energy nodal solution to \eqref{H}.
 	Then $U_p$ has exactly two nodal zones and its Morse index is 2.
 	\\
 	Further, up to rotation, $U_p$ is symmetric w.r.t.~the $x$ axis  (i.e. $U_p(x,-y)=U_p(x,y)$) and nonincreasing w.r.t. the polar variable in the semicircle $B\cap \{(x,y) : y>0\}$.
 	 \end{proposition}
 
 Of course one can repeat the same arguments on the subset of $H^1_{0}(B)$ made up by {\it radial functions}, thus ending with a {\it least energy nodal radial solution} which has exactly two nodal zones and {\it radial} Morse index equal to $2$. By the uniqueness result in \cite{NN}, we infer that the least energy nodal radial solution is nothing else that the radial solution with 2 nodal zones, to which we  we will refer as $u^{\ast}_p$ in the following.
 Next, since in \cite{AG-sez2-bis} it has been proved that the Morse index of any nodal radial solution is greater than 4, it follows that $U_p\neq u^{\ast}_p$  for every $p$.
 
 With the aim of producing other nonradial solutions, we introduce the so called {\it $n$-invariant functions}, studied in  \cite{GI} in the Lane-Emden case, and also in \cite{AG-N=2} in the H\'enon case (for large values of $p$).
 Precisely we denote by $H^1_{0,n}$ and   $\mathcal N^{\nod}_{p,n}$ the subsets of 
$H^1_{0}(B)$ and $\mathcal N^{\nod}_p$ made up by that functions which are invariant for reflection across the horizontal axis (i.e. even w.r.t.~$y$) and  for rotations of an angle $2\pi/n$.
They can be easily described by using the polar coordinates $(r,\theta)\in [0,\infty)\times[-\pi,\pi]$ defined by the relation $(x,y)= (r\cos\theta, r\sin\theta)$, by means of
\begin{align*} 
H^1_{0,n} &  : =    \left\{u\in    H^1_{0}(B)  \, : \, u(r,\theta) \hbox{ is even and } {2\pi}/n  \hbox{ periodic  w.r.t. } \theta, \, \hbox{ for every } r\in (0,1) \right\}, \\
 \mathcal N^{\nod}_{p,n} & :=  \mathcal N^{\nod}_p \cap H^1_{0,n} .
\end{align*}
Since of course also $H^1_{0,n}$  is compactly embedded in $L^p(B)$ for every $p>1$, for every integer $n$ $\mathcal E_ p(u)$ attains its minimum on  $\mathcal N^{\nod}_{p,n}$ at a nontrivial function $U_{p,n}$, which is a weak and also classical solution to \eqref{H}, and we call  {\it least energy nodal $n$-symmetric solution}.
By Proposition \ref{comeBWW} the least energy nodal solution belongs to $H^1_{0,1}$ and so it coincides with $U_{p,1}$. In particular $U_{p,1}$ is nonradial.
\\
To understand whether $U_{p,n}$ is radial or not when $n\ge 2$, we make use of the {\it $n$-Morse index}, i.e. the maximal dimension of a subspace of $H^1_{0,n}$ in which the quadratic form $\mathcal Q_{u}$ \eqref{Qp} is negative defined, or equivalently, the number of negative eigenvalues of the linearized operator $L_{u}$, according to \eqref{eigenvalue-problem}, which have corresponding eigenfunction in $H^1_{0,n}$. 
We refer hereafter to $m_n(u)$ as the $n$-symmetric Morse index of a solution $u$. 

Repeating the arguments of \cite{BW} in the $n$-invariant functional space  (see also \cite{AG-N=2}) one can see that 
\begin{lemma}\label{n-mi-unod}
The $n$-Morse index $U_{p,n}$ 
	 is equal to 2 for every integer $n$ and  $p>1$.
\end{lemma}
	
For what concerns the least energy nodal radial solution $u^{\ast}_p$, its $n$ Morse index can be computed starting from the singular eigenvalues of \eqref{radial-singular-problem-LE}  thanks to the analysis performed in \cite{AG-sez2}.

\begin{lemma}\label{n-mi-2}
For every positive integer $n$ and $p>1$ 
	\begin{equation}\label{n-morse-formula}
m_n(u^{\ast}_p)=2+\sum_{i=1}^2\left[ \frac{1}{n}\left\lceil \frac{2+\a}{2}\sqrt{- \nu_i(p)} -1\right\rceil \right] .
	\end{equation}
	Here $\nu_1(p)$ and $\nu_2(p)$ are the only negative eigenvalues of \eqref{radial-singular-problem-LE} related to \\ $W^{\ast}_p(t)=p\left(\frac{2+\a}{2}\right)^{2}\left|u^{\ast}_p(t^{\frac{2}{2+\a}})\right|^{p-1}$. 
\end{lemma}
\begin{proof}
 \cite[Corollary 4.11]{AG-sez2}  yields that the $n$-Morse index of a radial solution is 
\[
m_n(u^{\ast}_p)=\sum\limits_{i=1}^{2}\sum\limits_{j=0}^{\lceil  \frac{2+\a}2\sqrt{-\nu_i(p)}-1\rceil } N_j^{n} 
\]
where $N^n_j$ stands for the number of linearly independent  eigenfunctions related to the $j^{th}$ eigenvalue  of the Laplace Beltrami operator on the sphere $\mathbb S_1$ which  are $n$-invariant.
Recalling that the Laplace Beltrami eigenfunctions are linear combinations of  $\cos j\theta$ and $\sin j\theta$, one sees that $N_j^n=1$ if $j=0$ or is a multiple of $n$ (since in that case $\cos j\theta$ is $n$-invariant), and zero otherwise.
	Hence for $i=1,2$ fixed, we have 1  eigenfunction related to $j=0$ (by \eqref{nl<k-general-H}, \eqref{num>k-general-H}), and then another 
	eigenfunction for every index $j=hn$  where $h$ is an integer such that $1\le  h \le   \frac{1}{n}\left\lceil \frac{2+\a}{2}\sqrt{- \nu_i(p)}-1\right\rceil$, which totally gives $1+  \left[ \frac{1}{n}\left\lceil \frac{2+\a}{2}\sqrt{- \nu_i(p)} -1\right\rceil \right]$ independent eigenfunctions in $H_{0,n}$ and concludes the proof.
\end{proof}

Lemmas \ref{n-mi-unod} and \ref{n-mi-2}, together with the knowledge of the asymptotic Morse index of $u^\ast_p$ stated by Theorem \ref{mi-p=1}, yield that some of the least-energy $n$-symmetric nodal solutions are nonradial for $p$ close to 1.

\begin{corollary}\label{unod-nonradial}
	There exists $\bar p = \bar p(\alpha)>1$ such that $U_{p,n}$ are nonradial when $p\in (1, \bar p)$ and $n=1, \dots \left\lceil\frac{2+\a}{2}\beta-1\right\rceil$, where 
	$\beta \approx 2,\!305$ is the fixed number mentioned in Remark \ref{rem:N=2}.
\end{corollary}
\begin{proof}
Since $U_{p,n}$ is a least energy nodal solution and $H^1_{0,\rad}\subset H^1_{0,n}$, if it is radial  it coincides with $u^\ast_{p}$, up to the sign. So by Lemma \ref{n-mi-unod} $U_{p,n}$ is nonradial whenever $m_n(u^\ast_p)>2$, which in turn,  by Lemma \ref{n-mi-2},  holds true provided that  $n\le \left\lceil \frac{2+\a}{2}\sqrt{- \nu_1(p)} -1\right\rceil $, or equivalently   $n<\frac{2+\a}{2}\sqrt{- \nu_1(p)}$.
Eventually the claim follows since $\frac{2+\a}{2}\sqrt{- \nu_1(p)}\to \frac{2+\a}{2}\beta$  by \eqref{nup=1}.
\end{proof}

Nothing assures so far that this construction really produces $\left\lceil\frac{2+\a}{2}\beta-1\right\rceil$ distinct nonradial nodal solutions to \eqref{H}: in principle $U_{p,n}$ could coincide. Further nothing has been said about the other $n$-invariant solutions with larger $n$.
We find an answer to  these questions by inspecting the asymptotic profile of $U_{p,n}$  when $p\to 1$.
This issue has, in the author's opinion, interest for itself and will be the subject of next subsection.

\subsection{Asymptotic profile of  $n$-symmetric solutions}\label{ss:sym-asympt}

Here we describe the profile  of $U_{p,n}$ when $p\to 1$. 
First we obtain a bound for $\|U_{p,n}\|_{\infty}^{p-1}$ which ensures, via Lemma \ref{lem:conv-eigenv}, that $U_{p,n}/\|U_{p,n}\|_{\infty}$ converges to an eigenfunction of \eqref{prima-autof-weight}. Next we see that these limit eigenfunctions are nonradial and distinct one from another for $n=1, \dots \left\lceil\frac{2+\a}{2}\beta-1\right\rceil$, 
 while for the subsequent values of $n$ they coincide with the second radial eigenfunction. 

Obtaining a bound for $\|u\|_{\infty}^{p-1}$ when $u$ is a $n$-invariant solution  is far harder than in the radial setting, because it is not known a-priori where the extremal points accumulate when $p\to 1$. We show that a sufficient condition  is that the $n$-Morse index is bounded from above.
Such condition is certainly satisfied by the least energy $n$-invariant solutions thanks to Lemma \ref{n-mi-unod}.
\\
To this aim  we focus on the circular sector
\[ S_n=\{(r,\theta) \, : \, 0<r<1 , \ 0<\theta<\pi/n\},\]
and introduce some notations and preliminary materials.
The boundary of $S_n$ decomposes as $\partial S_n=\{O, A, B\}\cup\Gamma_1\cup \Gamma_2\cup \Gamma_3 $ where $O$ is the origin, and in standard coordinates $A=(\cos\frac{\pi}{n}, \sin\frac{\pi}{n} )$, $B=(1,0)$ and 
\[\begin{array}{cc}
\multicolumn{2}{c} { \Gamma_1=\left\{(x,y) \, : \, x^2+y^2=1 , \ \cos\frac{\pi}{n} < x < 1 , \, 0< y < \sin\frac{\pi}{n} \right\},} \\[4pt]
\Gamma_2= \left\{(x,y) \, : \, 0<x<1, \ y=0 \right\}, &  \Gamma_3  = \left\{(x,y) \, : \, \frac{y}{x} = \tan\frac {\pi}{n} , \ 0< x < \cos\frac{\pi}{n}\right\}.
\end{array}\]
Of course when $n=2$ we mean $\Gamma_3  = \left\{(x,y) \, : \, x=0, 0<y<1 \right\}$.
Next we set 
\begin{align*}
\widetilde S_n & = \overline{S}_n \cap B = S_n\cup \Gamma_2\cup\Gamma_3\cup \{O\}, \\
\mathcal C_n &= \{ v \in C(\overline{ S}_n )  \cap C^1(\widetilde S_n) \, : \,  v=0 \text{ on } \Gamma_1 , \,  \partial_{\nu} v =0 \text{ on } \Gamma_2\cup \Gamma_3\},
\end{align*}
where $\nu$ stands for the outer normal vector to the boundary of $S_n$.
Now  the restriction to $S_n$ of any function in $H^1_{0,n}\cap C^1(B)$ belongs to $\mathcal C_n$, and viceversa any function in $\mathcal C_n$ can be extended (by symmetry and periodicity) to a function in $H^1_{0,n}\cap C^1(B)$. In particular $u_p\in H^1_{0,n}$ is a classical solution of \eqref{H} if and only if its restriction to $S_n$ solves the mixed boundary problem
\begin{equation}\label{eq:spicchio}
\begin{cases}
-\Delta u = |x|^{\a} |u|^{p-1} u & \text{ in } S_n ,\\
u= 0 & \text{ on } \Gamma_1 , \\
\partial_{\nu} u =0 & \text{ on } \Gamma_2\cup \Gamma_3 ,
\end{cases}
\end{equation}
and it is elementary to check that 
\begin{lemma}\label{n-mi-unod-spicchio}
Let $u\in H^1_{0,n}$ be a solution of \eqref{H} whose $n$-Morse index is $m$. Then the eigenvalue problem
	\begin{equation}\label{e-p:spicchio}
	\begin{cases}
	-\Delta w = \left(p|x|^{\a} |u|^{p-1} +\mu \right) w& \text{ in } S_n ,\\
	w= 0 & \text{ on } \Gamma_1 , \\
	\partial_{\nu} w =0 & \text{ on } \Gamma_2\cup \Gamma_3 .
	\end{cases}
	\end{equation}
	has exactly $m$ negative eigenvalues.
	Equivalently, the maximal dimension of a subspace of $\mathcal C_n$ where the quadratic form 
	\begin{equation}\label{q-spicchio}
	\mathcal Q_u (w)= \int_{S_n} \left(|\nabla w|^2 - p|x|^{\a} |u|^{p-1} w^2\right) dx
	\end{equation}
	is negative defined is exactly $m$.
	\end{lemma}

The bound for $\|u_{p}\|_{\infty}^{p-1}$ is obtained by a refined  blow-up argument which starts  from \eqref{eq:spicchio} and ends up contradicting Lemma \ref{n-mi-unod-spicchio}. The blow-up procedure can bring to different domains for the limit problem, specifically $\R^2$ or an half space, or an angle. 
In this perspective we point out  that the limit eigenvalue problems have infinite Morse index in the following sense.

\begin{lemma}\label{morse-limite-infinito} Let $\a\ge 0$ and $\Sigma\subset \R^2$ (to be specified later). We  consider the quadratic form
		\[Q^{\a}_{\Sigma}(\phi)=\int_{\Sigma} \left(|\nabla \phi|^2-|x|^{\a}\phi^2\right) dx .\]
	Then for every integer $k$ there exist $R>0$ and $k$ linearly independent functions with support contained in $\overline B_R$ which are zero on $\partial B_R \cap \Sigma$, are continuous on the closure of $\partial B_R \cap \Sigma$ and $C^1$ in its interior,  satisfy 
	\[Q^{\a}_{\Sigma}(\phi_{j})<0\quad \text{ as }j=1, \dots k\] 
	and
	\begin{enumerate}
		\item if $\Sigma=\R^2$, then $\phi_{j}\in H^1_{0,n}$.
		\item If $\Sigma$ is an half-plane of type $\{Q : Q\cdot P <0\}$,  then $\phi_j$ are symmetric with respect to the direction orthogonal to $P$ and there exists $\delta>0$ such that $\phi_j=0$  in a strip $\{Q  : Q\cdot P \ge -\delta\}$.
		\end{enumerate}
\end{lemma}
\begin{proof}
	We use the notations of Lemma \ref{lem:autov-peso} and define, in polar coordinates
	\begin{align*}
	\phi_{j}(r, \theta) & = \mathcal J_{\frac{2jn}{2+\a}}\left(\frac{Z_j}{R} r\right) \cos(jn\theta) \quad & \text{ in } B_R , 
	\end{align*}
	for $Z_j=z_1\left(\frac{2jn}{2+\a}\right)$.
They satisfy $-\Delta \phi_{j} = |x|^{\a}(\frac{Z_j}{R})^2 \phi_{j}$ pointwise on $B_R$  and vanish at $\partial B_R$.
Extending them to zero outside $B_R$, multiplying the equation by $\phi_{j}$ and integrating by parts on $B_R$ (taking advantage from the boundary condition on $\partial B_R$) gives
	\begin{align*}
Q^{\a}_{\R^2}(\phi_{j}) = \int_{B_R}\left( |\nabla \phi_{j}|^2-|x|^\a\phi_{j}^2\right) dx = \left(\left(\frac{Z_j}{R} \right)^2 -1\right) \int_{B_R}|x|^\a \phi_{j}^2 dx <0 
	\end{align*}
	for $j=1, \dots k$ provided that $R> Z_k$.
	So we have obtained the functions requested by item (1).
	\\
	Concerning the following item, since the quadratic form is invariant by rotation it is sufficient to make the proof only in one particular set of type (2) 
	so we fix $P=(1,0)$ and $\Sigma = \{ (x,y) : x<0\}$.
	 Next we take a cut-off function $\xi\in C^1(-\infty,0]$ so that
\[ 0\le \xi\le 1 , \quad \xi (t)= \begin{cases} 1 & \text{ as } t \le - 2\delta , \\ 0 & \text{ as } -\delta \le t \le 0 ,\end{cases} \quad -2/\delta \le \xi'(t) \le 0\ \text{ as } -2\delta \le t \le -\delta  \]
and define (in standard coordinates)
\begin{align*}
\psi_{j}(x, y) & =\phi_{j}(x,y)\,  \xi(x) \quad & \text{ in } B_R \cap\{(x,y) : x<0\} .
\end{align*}
 It is clear that $\psi_{j}$  is zero on $\partial B_R$ and when $-\delta \le x $, and even w.r.t.~y. Besides $Q^{\a}_{\Sigma}(\psi_{j})\to Q^{\a}_{\Sigma}(\phi_j) = \frac{1}{2}Q^{\a}_{\R^2}(\phi_{j}) < 0$ as $\delta\to 0$, concluding the proof of item (2). Indeed it is clear that $\int_{\Sigma}  |x|^\a\left(\phi_{j}\xi\right)^2 dx \to \int_{\Sigma} |x|^\a\phi_{j}^2 dx$ as $\delta \to 0$, moreover 
 since $\phi_{j}=0 $ at $x=0$, using its  $C^1$ regularity and  the properties of $\xi$ gives
\begin{align*}
\left|\int_{\Sigma} \left(|\nabla \left(\phi_{j}\xi\right)|^2 - |\nabla \phi_j|^2 \right) dx \right| & \le  \int_{\Sigma\cap B_R} |\nabla \phi_{j}|^2 \left(1-\xi^2\right) dx \\
 + \int_{\Sigma\cap B_R} \phi_{j}^2|\xi'|^2 dx+ 2 \int_{\Sigma\cap B_R} |\phi_{j}| \,|\nabla\phi_{j}| \, |\xi| \, |\xi'|\,  dx  & \le C \, {\mathrm{meas}}\left(B_R\cap\{- \delta\le x \le 0\}\right) .
\end{align*}
	\end{proof}
\begin{remark}\label{rem:a=0-trasl}
	For $\a=0$ the quadratic form in $Q^{0}_{\Sigma}$ is invariant also for translation, so Lemma \ref{morse-limite-infinito} continue to hold also for shifted sets of type $\Sigma+ P_0$, for every $P_0\in\R^2$.
\end{remark}

We are now in the position to prove the $L^{\infty}$ estimate.

\begin{proposition}\label{prop:10.4}
	Let $p_k\to 1$ and $u_k\in H^1_{0,n}$ a sequence of solutions of \eqref{H} with $p=p_k$.
	If $m_n(u_k)\le m <\infty$, then there exists a 	constant $C$ such that $\|u_k\|_{\infty}^{p_k-1}\le C$ for large $k$.
\end{proposition}
\begin{proof}
	We argue by contradiction and take that $\|u_k\|_{\infty}^{p_k-1}\to\infty$ along a subsequence, that we still denote by  $p_k$.
By the symmetry of $u_k$ there is $P_k\in S_n$  where $u_k$ achieves its extremal value, that we can take to be a maximum w.l.o.g.
	Up to another sequence $P_k\to \bar P \in \bar S_n$, and we argue differently according to the location of $\bar P$.
	\\
	Before going on, let us introduce some notations and make some general considerations. We write $X=(x,y)$ for a generic point in $\R^2$, $P_k=(x_k,y_k)$, $\bar P = (\bar x, \bar y)$.
	Whenever $\bar P\neq 0$ we take as  scaling parameter
	\begin{equation}\label{Mk} 
	 M_k = \|u_k\|_{\infty}^{\frac{p_k-1}{2}} |\bar P|^{\frac{\a}{2}},
	 \end{equation}
	 and  introduce the change of variables 
	 \begin{equation}\label{change}
	  X' = P_k + \frac{X}{M_k},  \quad 
	 \widetilde u_k(X) = \frac{ u_k (X')}{\|u_k\|_{\infty}}  \ \text{ as } X \in \Sigma_k= \{ X\in \R^2  : \, X' \in S_n\}.
	 \end{equation}
	The regular part  of the boundary of $\Sigma_k$ is made up by the curves 
	\begin{align} \label{Gamma1}
	\Gamma_{1,k} = \big\{ X\!=\!(x,y) : & \, \left|X/M_k + P_k\right|=1 , \ 0< \frac{y+ y_k \, M_k}{x+ x_k \, M_k} < \tan \frac{\pi}{n} \big\} \\
	\label{Gamma2}
	\Gamma_{2,k} = \big\{ X\!=\!(x,y) : & \  -M_kx_k < x < M_k(1-x_k) , \ y= -M_k y_k \big\} \\
		\label{Gamma3} 
	\Gamma_{3,k} = \big\{ X\!=\!(x,y) : & \ \frac{y+ y_k \, M_k}{x+ x_k \, M_k} = \tan \frac{\pi}{n}, \ -M_kx_k < x < M_k \left(\cos \frac{\pi}{n} -x_k\right) , \\ \nonumber
	& \  -M_ky_k < y < M_k \left(\sin \frac{\pi}{n} -y_k\right)\big\} ,
	\intertext{ if $n\neq 2$, otherwise } \nonumber
	\Gamma_{3,k} = \big\{ X\!=\!(x,y) : & \ x=-M_k x_k , \ -M_ky_k < y < M_k \left(1  -y_k\right)\big\} .
	\end{align}
	 Next $\wtu_k$  solves
	\begin{align}\label{eq:wtu}
	\begin{cases}
-\Delta \widetilde u_k =  \rho_k |\wtu_k|^{p_k-1} \wtu_k  & \text{ in } \Sigma_k , \\
\wtu_k = 0 & \text{ on } \Gamma_{1,k} , \\
\partial_{\nu} \wtu_k =0 & \text{ on } \Gamma_{2,k}\cup \Gamma_{3,k} ,
	\end{cases}
	\end{align}
for 
\begin{equation}\label{rho} 
\rho_k(X)= \frac{\left|X'\right|^{\a}}{|\bar P|^\a} = \frac{\left|P_k + X/M_k\right|^{\a}}{|\bar P|^\a} .
\end{equation}
Notice that $M_k\to \infty$ and $\rho_k\to 1$ locally uniformly unless $\bar P=O$.
\\
Further $\wtu_k(O) = 1 = \|\wtu_k\|_{\infty}$ for every $k$ and by standard elliptic estimates if $\Sigma$ is any open subset of $\R^2$ such that $\Sigma\subset \Sigma_k$ for large $k$, then $\wtu_k$ converges weakly in $H^1_{\loc}(\Sigma)$ and  in $C_{\loc}(\Sigma)$ to a function $\wtu$ which solves
\begin{equation}\label{eq:limite-tutto} 
-\Delta \widetilde u =  \wtu \quad \text{ in } \Sigma 
\end{equation}
in weak sense (and therefore also in classical sense).
	Indeed for every $\varphi\in C^{\infty}_0(\Sigma)$, taken $k$ so large that the support of $\varphi$ in contained in $\Sigma_k$, we have
	\begin{align*}
	0 & = \int_{\Sigma} \left( \nabla \wtu_k \nabla\varphi - \rho_k |\wtu_k|^{p_k-1} \wtu_k \varphi \right) dX \\
	&	= \int_{\Sigma} \left( \nabla \wtu_k \nabla\varphi - \wtu_k \varphi \right) dX +
	\int_{\Sigma}\left(\rho_k -1\right)|\wtu_k|^{p_k-1} \wtu_k \, \varphi dX + 	\int_{\R^2}\left(|\wtu_k|^{p_k-1} -1\right)\wtu_k \, \varphi dX
	\end{align*}
	where the first integral goes to $\int_{\Sigma} \left( \nabla \wtu \nabla\varphi - \wtu \varphi \right) dX$ by the weak convergence of $\wtu_k$, the second one vanishes by the locally uniform convergence of $\rho$  and the boundedness of $\wtu_k$, and the third one vanishes too as it can be estimated by
	\begin{align*}
	\left|\int_{\Sigma}\left(|\wtu_k|^{p_k-1} -1\right)\wtu_k \, \varphi dX \right|\underset{\eqref{per-dopo}}{=}	(p_k-1)\int_{\Sigma} \int_0^1|\wtu_k|^{t(p_k-1)} dt \, |\wtu_k| \,\log|\wtu_k| \,  | \varphi| dX \\ \le c (p_k-1)\int_{\Sigma}  | \varphi| dX .
	\end{align*}	
If in addition $\Sigma$ can be taken such  that $O\in \Sigma$, then $\wtu$ is nontrivial because clearly $\wtu(O)=1$.
In that case the zero-set of $\wtu$ is made up by regular curves that may intersect only at some isolated points, see, for instance, \cite{CF}. Therefore 
\begin{equation}\label{civoleva}
\int_\Sigma p_k \rho_k |\wtu_k|^{p_k-1} \varphi dX \to \int_\Sigma  \varphi dX 
\end{equation}
for every function $\varphi\in C_0(\Sigma)$ .

	\
	
	{\it Case 1: $\bar P\in S_n$.}
	In this case  looking at \eqref{Gamma1}-\eqref{Gamma3} one sees that $\Sigma_k$ invades $\R^2$, so that $\wtu$ is a nontrivial solution of \eqref{eq:limite-tutto} with $\Sigma=\R^2$.
On the other hand by Lemma \ref{morse-limite-infinito}, item (1)  there exist at least $m+1$ linearly independent $n$-invariant functions $\phi_{j}$ with compact support such that $Q^0_{\R^2}(\phi_{j})<0$. 
So \eqref{civoleva} implies
\begin{align*}
\int_{\Sigma_k} \left(|\nabla \phi_{j}|^2 - p_k\rho_k|\wtu_k|^{p_k-1} \phi_{j}^2\right) dX < 0 .
\end{align*}
Next we come back according to the change of variables \eqref{change} and define the functions $w_j(X')= \phi_{j}(X)$ as $X'\in S_n$. They belong to the space $\mathcal C_n$ for large $k$ since their support is contained in a ball of radius $R/M_k$ centered at $P_k$ with $P_k\to \bar P\in S_n$.  Moreover they are linearly independent and satisfy
\begin{align*}
\int_{S_n} \left(|\nabla w_j(X')|^2 - p_k |X'|^{\a}|U_{p_k, n}(X')|^{p_k-1} w_j^2(X')\right) dX' = \\
\int_{S_n} \left( |\nabla \left(\phi_{j}(M_k(X'\!-\!P_k))\right)|^2 - p_k M_k^2 \frac{|X'|^{\a}}{|\bar P|^{\a}}\left|\widetilde u_k\left(M_k(X '\!-\!P_k)\right) \right|^{p_k-1} \phi_j^2(M_k(X'\!-\!P_k)) \right)  dX'
\\ 
= \int_{\Sigma_k} \left(|\nabla \phi_{j}(X)|^2 - p_k \rho_k(X) \left|\widetilde u_k(X)\right|^{p_k-1} \phi_{j}^2(X)\right) dX  < 0,
\end{align*}
which contradicts Lemma \ref{n-mi-unod-spicchio}.

\

{\it Case 2: $\bar P\in \Gamma_1$.} Different situations present depending if ${\mathrm{dist}}(P_k,\Gamma_1)\, M_k \to \infty$ or ${\mathrm{dist}}(P_k,\Gamma_1)\, M_k \to s>0$.
Indeed ${\mathrm{dist}}(P_k,\Gamma_1)\, M_k $ cannot vanish because of the elliptic regularity up to the boundary, see \cite[Case 2 in the proof of Theorem 1.1]{GS81a}. 

\medskip
{\it Case 2.a.} If  ${\mathrm{dist}}(P_k,\Gamma_1)\, M_k \to \infty$, looking at \eqref{Gamma1}-\eqref{Gamma3} one sees that  $\Sigma_k$ invades $\R^2$ and the conclusion follows as in Case 1.

\medskip
{\it Case 2.b.} If  ${\mathrm{dist}}(P_k,\Gamma_1)\, M_k \to s>0$, then for every $X\in \R^2$ we have 
\[{M_k}\left(\left|\frac{X}{M_k} + P_k\right|^2-1\right)=  \frac{|X|^2}{M_k} + 2 X\cdot P_k - (1+|P_k|) M_k \,\mathrm{dist}(P_k, \Gamma_1)
\to 2 (X-s\bar P)\cdot \bar P .\]
So recalling \eqref{Gamma1} one sees that 
 the curve $\Gamma_{1,k}$ goes to the straight line $\{ X \, : \, (X-s\bar P)\cdot \bar P =0 \}$, and $\Sigma_k$ invades the half-plane $\Sigma =\{ X\, : \, (X-s\bar P)\cdot \bar P <0 \}$.
 Because $s>0$, then $O\in \Sigma$ and $\wtu$ is a nontrivial solution to \eqref{eq:limite-tutto}.
 
  Now by Lemma \ref{morse-limite-infinito}, item (2) and Remark \ref{rem:a=0-trasl} there are $m+1$ linearly independent functions $\phi_{j}$ which are zero outside $B_R\cap \{X: (X-s\bar P)\cdot \bar P \le -\delta \}$ such that $Q^0_{\Sigma}(\phi_j)<0$. In particular, for $k$ sufficiently large, their support is contained in $\Sigma_{k}$ and by the weak convergence of $p_k \rho_k |\wtu_k|^{p_k-1}$  we infer that
\[
\int_{\Sigma_k} \left(|\nabla \phi_{j}|^2 - p_k\rho_k|\wtu_k|^{p_k-1} \phi_{j}^2\right) dX <0 .
\]
Eventually the functions $w_j(X')= \phi_{j}(X)$ as $X'\in S_n$ belong to the space $\mathcal C_n$ for large $k$, because their supports do not touch the boundary of $S_n$. Moreover they  are linearly independent and by the change of variables \eqref{change} one sees that 
\begin{align*}
\int_{S_n} \left(|\nabla w_j|^2 - p_k |X'|^{\a}|u_k|^{p_k-1}w_j^2\right) dX' 
=  \int_{\Sigma_k} \left(|\nabla \phi_{j}|^2 - p_k\rho_k|\wtu_k|^{p_k-1} \phi_{j}^2\right) dX <0 ,
\end{align*}
which contradicts Lemma \ref{n-mi-unod-spicchio}.

\

{\it Case 3: $\bar P\in \Gamma_2\cup\Gamma_3$.}
We only consider the case $\bar P\in \Gamma_2$, as the other one can be handled similarly. 
Now $\bar P=(\bar x,0)$ for some $0<\bar x<1$ and $\mathrm{dist}(P_k, \Gamma_2)= y_k$ for large $k$.
Again the limit set for $\Sigma_k$ changes according if either $y_k M_k \to \infty$ or to some $ t\ge  0$.

\medskip
{\it Case 3.a.} If $y_k  M_k \to \infty$,  $\Sigma_k$ invades $\R^2$ and the conclusion follows as in Case 1.

\medskip
{\it Case 3.b.} If $y_k  M_k \to t\ge 0$, looking at \eqref{Gamma2} one sees that the segment $\Gamma_{2,k}$ goes to the straight line $y=-t$, and $\Sigma_k$ invades the half-plane $\{(x,y) : y>-t\}$. 
So instead of  $S_n$ we focus into a sector of amplitude $2\pi/n$, namely 
\[ S^{\sharp}_n = \{ X=(x,y) \in \R^2 : (x,y)  \text{ or } (x, -y) \in S_n \} ,\] 
and we slightly modify the scaling by taking 
\begin{equation}\label{change2} \begin{split} 
X'=(x', y') \ \text{ given by } x'= x_k + \frac{x}{M_k} , \  y'=  \frac{y}{M_k} , \ \text{ as } X=(x,y) \\
\wtu_k(X)= \frac{u_k(X')  }{\|u_k\|_{\infty}} \qquad \text{ for } X \in \Sigma^{\sharp}_k= \{ X : X' \in S^{\sharp}_n\}
\end{split}\end{equation}
instead of \eqref{change}.
Minor changes to the previous arguments yield that $\Sigma^{\sharp}_k$ covers $\R^2$ and $\wtu_k$ converges in $C_{\loc}(\R^2)$ to a solution of \eqref{eq:limite-tutto} in $\R^2$. The locally uniform convergence ensures that $\wtu(0,t)=1$ (because $\wtu_k(0, y_kM_k)=1$ with $y_kM_k\to t$), so that $\wtu$ is not identically zero and then also \eqref{civoleva} holds true.
Thank to this  the same arguments used in Case 1 give that
\begin{align*}
\int_{S^{\sharp}_n} \left(|\nabla w_j|^2 - p_k |X'|^{\a}|u_k|^{p_k-1} w_j^2\right) dX'  < 0,
\end{align*}
where $w_j (X')= \phi_{j}(X)$, and $\phi_j$ are the functions produced in Lemma \ref{morse-limite-infinito}, item (1).
In particular, due to the modified change of variables \eqref{change2}, both $u_k$ and $w_j$ are even w.r.t. the $y' $ variable, therefore also 
\begin{align*}
\int_{S_n} \left(|\nabla w_j|^2 - p_k |X'|^{\a}|u_k|^{p_k-1} w_j^2\right) dX'   < 0,
\end{align*}
and $\partial_y w_j(x',0)=0$, i.e.  $\partial_{\nu}w_j=0$ on $\Gamma_2$. 
So the functions $w_j$, restricted to $S_n$, belong to $\mathcal C_n$ (because their support does not touch $\Gamma_1$ or $\Gamma_3$ if $k$ is large enough) and provide a contradiction with Lemma  \ref{n-mi-unod-spicchio}.

\

{\it Case 4: $\bar P=A$ or $B$.}
We take $\bar P=B$, as the case $\bar P=A$ is similar. Now 
\[{\mathrm{dist}}(P_k, \partial S_n)= \min\left\{ {\mathrm{dist}}(P_k, \partial \Gamma_{1}) \, , \, {\mathrm{dist}}(P_k, \partial \Gamma_{2}) \right\}
.\]
If ${\mathrm{dist}}(P_k, \partial S_n)\, M_k \to \infty$ the conclusion follows as in Case 1. Otherwise we have to distinguish between different occurrencies:
\begin{enumerate}[\it{Case 4.}a:]
	\item ${\mathrm{dist}}(P_k, \partial \Gamma_{1})\, M_k\to s > 0$ and ${\mathrm{dist}}(P_k, \partial \Gamma_{2}) \, M_k \to \infty$,
\item    ${\mathrm{dist}}(P_k, \partial \Gamma_{1})\, M_k \to \infty $ and ${\mathrm{dist}}(P_k, \partial \Gamma_{2}) \, M_k \to t\ge 0$,
	\item ${\mathrm{dist}}(P_k, \partial \Gamma_{1}) \, M_k \to s > 0$ and ${\mathrm{dist}}(P_k, \partial \Gamma_{2}) \, M_k \to t \ge 0$.
\end{enumerate}
Indeed the occurrence ${\mathrm{dist}}(P_k, \partial \Gamma_{1})\, M_k \to  0$ can not happen by the considerations in \cite{GS81a}.
Case 4.a and 4.b can be ruled out as we have done for  Cases 2.b and  3.b, respectively.

As for Case 4.c, using the change of variables \eqref{change2}  the sets $\Sigma^{\sharp}_k$ cover the half-plane $\Sigma=\{(x,y) : x < s \}$ and the functions $\wtu_k$ converge in $C^1_{\loc}(\Sigma)$ to a solution of \eqref{eq:limite-tutto} on $\Sigma$, which is nontrivial since $u_k(0, y_kM_k) =1 $ with $(0, y_kM_k)\to (0, t)\in \Sigma$ as $s>0$.
Next the functions $\phi_j$ produced in Lemma \ref{morse-limite-infinito}, item (2) have support compactly contained in  $\Sigma^{\sharp}_k$ and so 
\[ \int_{\Sigma^{\sharp}_k}  \left(|\nabla \phi_j|^2 - p_k \rho_k |\wtu_k|^{p_k-1} \phi_j^2\right) dX  < 0 \]
by the weak convergence \eqref{civoleva}.
Taking $w_j (X')= \phi_{j}(X)$, one can easily see that 
\begin{align*}
\int_{S^{\sharp}_n} \left(|\nabla w_j|^2 - p_k |X'|^{\a}|u_k|^{p_k-1} w_j^2\right) dX'  < 0 .
\end{align*}
But, due to the modified change of variables \eqref{change2}, $w_j$ are even w.r.t. the $y' $ variable, and certainly the same holds for $u_k$.
Therefore  also 
\begin{align*}
\int_{S_n} \left(|\nabla w_j|^2 - p_k |X'|^{\a}|u_k|^{p_k-1} w_j^2\right) dX'  < 0 .
\end{align*}
and by symmetry on $\Gamma_2 \subset \{(x', y') : y'=0\}$ we have $\partial_{\nu} w_j = \partial_y w_j(y', 0)= 0$.
Eventually  the restriction of $w_j$ to $\{(x',y') : y'>0\}$ belongs to ${\mathcal C}_n$, because its support does not touch either $\Gamma_1$ or $\Gamma_3$ if $k$ is taken sufficiently large, and this concludes this part of the proof. 

\medskip

{\it Case 5: $\bar P=O$.}
Here we need a different scaling parameter because $M_k$ given in \eqref{Mk} is constantly  $0$.

{\it Case 5.1: $\bar P=O$ and $|P_k|^{2+\a} \|u_k\|_{\infty}^{p_k-1} \to \infty$.} 
We take 
\[ M_k = |P_k|^{\frac{\a}{2}} \|u_k\|_{\infty}^{\frac{p_k-1}{2}} ,
\] 
so that $M_k = \frac{ |P_k|^{\frac{2+\a}{2}} \|u_k\|_{\infty}^{\frac{p_k-1}{2}}}{|P_k|}\to \infty $.
Next we use the same change of variables \eqref{change}, but with a different value for $M_k$.
In that way $\wtu_k$ solves a problem of type \eqref{eq:wtu} for 
\begin{equation}\label{rho5}
\rho_k(X)= \left(\frac{|X'|}{M_k |P_k|}\right)^{\a} = \left(\frac{| M_k P_k + X|}{M_k |P_k|}\right)^{\a},
\end{equation} 
which converges to $1$ locally uniformly because also $M_k |P_k| \to \infty$.
\\
Concerning the limit set for $\Sigma_k$, it changes depending if $\mathrm{dist}(P_k , \partial \Sigma_k) \, M_k$ is bounded or not.
If $\mathrm{dist}(P_k , \partial \Sigma_k)\, M_k \to \infty$, then $\Sigma_k$ invades $\R^2$ and  we can conclude as in Case 1.

Otherwise if $\mathrm{dist}(P_k , \partial \Sigma_k) \, M_k  \to t \ge 0$, we remark that 
\[ \mathrm{dist}(P_k , \partial \Sigma_k)= \min\left\{ \mathrm{dist}(P_k , \Gamma_{2,k}) , \mathrm{dist}(P_k , \partial \Gamma_{3,k})\right\} .\]
To fix ideas we take that  $\mathrm{dist}(P_k , \partial \Sigma_k) = \mathrm{dist}(P_k , \Gamma_{2,k}) =y_k$ along a subsequence (the opposite case can be dealt in similarly).
In the present situation $y_k M_k$ is bounded and $|P_k| M_k \to \infty$, therefore $x_k M_k\to \infty$.
Consequently $\Sigma_k$ goes to the half-space $\{y> -t\}$, and one can reason as in  Case 4.b.

\medskip 

{\it Case 5.2: $\bar P=O$ and $|P_k|^{2+\a} \|u_k\|_{\infty}^{p_k-1}$ is bounded.}  
\\
We chose as a scaling parameter
\[\mu_k= \|u_k\|_{\infty}^{\frac{p_k-1}{2+\a}},\]
and define 
\begin{align}\label{change-ultimo}
\widehat u_k(X) & = \frac{1}{\|u_k\|_{\infty}} u_k (X/\mu_k) \quad \text{ in } B_k = \{X : |X|<\mu_k\} ,
\end{align}
which solves
\[\begin{cases}
-\Delta \widehat u_k = |X|^{\a} |\widehat u_k|^{p_k-1} \widehat u_k & \text{ in } B_k , \\
\widehat u_k=0 & \text{ on } \partial B_k .
\end{cases}\]
Now $\widehat u_k(\mu_kP_k)= 1 = \|\widehat u_k\|_{\infty}$ and  $B_k$ invades $\R^2$, therefore one can see that also in this case  $\widehat u_k$ converges locally uniformly to a  solution of 
\[ - \Delta \widehat u  = |x|^{\a} \widehat u \quad x\in \R^2.\]
Let us check that the function $\widehat u$ is nontrivial. Since $\mu_k\, P_k$ is bounded we can assume that it converges to some point $Q_0$, and 
by the locally uniform convergence $\widehat u(Q_0)=1$.
\\
Eventually, take the $m+1$ linearly independent functions $\phi_j$ produced in Lemma \ref{morse-limite-infinito}, step (1) with compact support such that $Q^{\a}_{\R^2}(\phi_j)<0$.
The convergence of $\widehat u_k$ yields that also
\[ \int_{B_k}\left( |\nabla \phi_j|^2 - p_k |X|^\a |w_k|^{p_k-1} \phi_j^2\right) dX < 0 .\]
Eventually defining $\psi_j(X)= \phi_j(\mu_k\, X)$ for $X\in B$, we see that  $\psi_j\in H^1_{0,n}$ (because the change of variable \eqref{change-ultimo} does not break the symmetries) and 
\begin{align*}
\int_{B} \left(|\nabla \psi_j|^2 - p_k |X|^{\a}|u_k|^{p_k-1} \psi_j^2\right) dX =\int_{B_k} \left(|\nabla \phi_{j}|^2 - p_k |X|^\a |w_k|^{p_k-1} \phi_j^2\right) dX <0 ,
\end{align*}
which contradicts the fact that the $n$-Morse index of $u_k$ is at most $m$.
\end{proof}

Thank to Lemma \ref{n-mi-unod} and Proposition \ref{prop:10.4}, we can apply  Lemma \ref{lem:conv-eigenv} to any sequence $U_{p_k,n}$ of least energy $n$-invariant nodal solutions with $p_k\to 1$ and deduce that  \eqref{limite} and  \eqref{weak-nod} hold for some eigenvalue $\mu_{j,i}$ of \eqref{prima-autof-weight} and related  eigenfunction $\omega_{j,i}$ (normalized so that $\|{\omega}_{j,i}\|_{\infty}=1$). Of course the index $(j,i)$ has to be selected in such a way that $ {\omega}_{j,i}\in H^1_{0,n}$.
So from the computation preformed in  Lemma \ref{lem:autov-peso} we deduce that 
\begin{align} \label{autov-peso-N=2}
\mu_{j,i} & = \left(\frac{2+\alpha}{2} \, z_{i}\left(\frac{2j}{2+\alpha}\right)\right)^2 ,\\
 \label{autof-peso-N=2}
{\omega}_{j,i}(r,\theta) & =  {\pm}\frac{1}{\|\jcal_{\frac{2j}{2+\alpha}}\|_{\infty}} \jcal_{\frac{2j}{2+\alpha}}\left(z_{i}\left(\frac{2j}{2+\alpha}\right)r^{\frac{2+\alpha}{2}}\right) \cos(j\theta), 
\end{align}
where the index $j$ can be $0$ or a multiple of $n$.
Here we have also used that each Bessel function attains is global extremum in its first nodal interval and that in dimension $N=2$ the eigenfunctions of the Laplace-Beltrami operator are of type $A\cos(j\theta)+B\sin (j\theta)$, so that they belong to $H^1_{0,n}$ only when $j=0$ or $j$ is a multiple of $n$ and $B$ is zero.
\\
Next we see that the minimality of $U_{p,n}$ implies that $\mu_{j,i}$ must be the second eigenvalue of \eqref{prima-autof-weight} in the space $H^1_{0,n}$, and so we 
 single out the exact value of $j$ and $i$. This gives Theorem \ref{n-asympt}

\begin{proof}[Proof of Theorem \ref{n-asympt}]
With a little abuse of notation, we write $p\to 1$ meaning any sequence	$p_k\to 1$.
We have already pointed out that $\|U_{p,n}\|_{\infty}^{p-1}$ and $ \bar U_p:=U_{p,n}/\|U_{p, n}\|_{\infty}$ converge respectively to an eigenvalue $\mu_{j,i}$ and an eigenfuntion $ \omega_{j,i}$ of \eqref{prima-autof-weight} described by \eqref{autov-peso-N=2} and \eqref{autof-peso-N=2}.
	It remains to check that  the values of $j$ and $i$ in \eqref{limite} are $n$ and $1$ (or $0$ and $2$)  if $n<\frac{2+\a}{2}\beta $ (or else $n> \frac{2+\a}{2}\beta $).
	We divide the proof of this fact  in several steps.
	
	\
	
	{\it Step 1:} $\mu_{j,i} > \mu_{0,1}$.
	\\
	First  we observe that $\bar u_p$  cannot go to the first eigenfunction $\omega_{0,1}$.
	Otherwise  \eqref{limite} assures that  $p\|U_{p,n}\|_{\infty}^{p-1}\to \mu_{0,1}<\mu_{h,\ell}$ for every $(h,\ell)\neq (0,1)$, since the first eigenvalue is simple.
	Letting $\mu$ be the second eigenvalue of \eqref{prima-autof-weight} in $H^1_{0,n}$, for every $w\in H^1_{0,n}$, $w\perp \omega_{0,1}$ we have
	\begin{align*}
	\mathcal Q_{U_{p,n}}(w)& =\int _B |\nabla w|^2dx - \int _B p |x|^\a|U_{p,n}|^{p-1} w^2 dx  \ge 
	\mu \int _B |x|^{\a}w ^2dx- \int _B p |x|^\a|U_{p,n}|^{p-1} w^2 dx \\
	&\ge  \left(\mu - p\|U_{p,n}\|_{\infty}^{p-1}\right) \int _B |x|^{\a}|w|^2dx  \ge  \e \int _B |x|^{\a}|w|^2dx  
	\end{align*}
	for some $\e>0$ when $p$ is close to $1$. 
	Hence in this case the $n$-Morse index of $U_{p,n}$ would be  at most 1, contradicting Lemma \ref{n-mi-unod}.
	
	\
	
	{\it Step 2:} $\mu_{j,i}\le \mu_{n,1}$.
	\\
	By the minimality of $U_{p,n}$, and observing that  $\mathcal E_ p (v)= \frac{p-1}{2(p+1)}	\int_B |\nabla v |^2 dx $ for any function $v\in \mathcal N^{\nod}_{p,n}$, we have 
	\begin{align*}
	\int_B |\nabla \bar U_{p,n}|^2 dx = \frac{1}{\|U_{p,n}\|_{\infty}^2} \int_B |\nabla U_{p,n}|^2 dx \le \frac{1}{\|U_{p,n}\|_{\infty}^2} \int_B |\nabla v|^2 dx 
	\end{align*}
  for every $v\in \mathcal N^{\nod}_{p,n}$.	So \eqref{limite} implies that
	\begin{align}\label{goal}
	0< \int_B |\nabla  \omega_{j,i}|^2 dx \le \liminf\limits_{p\to 1}\frac{1}{\|U_{p,n}\|_{\infty}^2} \int_B |\nabla v_p|^2 dx 
	\end{align}
	for every sequence  $v_p$ in $\mathcal N^{\nod}_{p,n}$.
	
	Next we define 
	\begin{align*}
	v_p  : =  A_p  \omega_{n,1} 
	\end{align*}
	where $  \omega_{n,1}$ is defined according to \eqref{autof-peso-N=2}, 
	and check that we can chose the constant $A_p>0$ in such a way that $v_p\in {\mathcal N}^{\nod}_{p,n}$ for every $p$.
	\\
	The support of the   positive/negative parts of $v_p$  are $B^{\pm} = \bigcup\limits_{i=0}^{n-1} \Sigma^{\pm}_i$ for
	$\Sigma^+_i=\big\{(r,\theta) \, : \, 0\le r\le 1, \,  \frac{2i-1}{2n}\pi \le\theta \le \frac{2i+1}{2n}\pi \big\}$ and $\Sigma^-_i=\big\{(r,\theta) \, : \, 0\le r\le 1, \, \frac{2i+1}{2n} \pi\le \theta \le \frac{2i+3}{2}\pi \big\}$,	so that by periodicity
	\begin{align*}
	\int_B |x|^{\a}|v_p^+|^{p+1} dx & =  A_p^{p+1} 	\int_B |x|^{\a}| \omega_{n,1}^+|^{p+1} dx =
	n A_p^{p+1} \int_{\Sigma_0^+} |x|^{\a}| \omega_{n,1}|^{p+1} dx \intertext{and by simmetry}
	&=   n A_p^{p+1} \int_{\Sigma_0^-} |x|^{\a}| \omega_{n,1}|^{p+1} dx = A_p^{p+1} 	\int_B |x|^{\a}| \omega_{n,1}^-|^{p+1} dx =
	\int_B |x|^{\a}|v_p^-|^{p+1} dx .
	\end{align*}
	On the other hand $v_p$ is an eigenfunction for \eqref{prima-autof-weight}  related to $\mu_{n,1}$, hence 
	\begin{align*}
	\int_B |\nabla v_p^{\pm}|^2 dx & =  \int_B \nabla v_p \nabla v^{\pm} dx = 
	\mu_{n,1}\int_B |x|^{\a} v_p v_p^{\pm} dx = 	\mu_{n,1}\int_B |x|^{\a} |v_p^{\pm}|^2 dx 
	\end{align*}
	and by virtue of the periodicity and simmetry of $\cos(n\theta)$ we get
	\begin{align*}
	\int_B |\nabla v_p^+|^2 dx & = A_p^{2} \mu_{n,1} \int_B |x|^{\a} | \omega^{+}_{n,1}|^2 dx = \frac{1}{2} A_p^{2} \mu_{n,1} \int_B |x|^{\a} | \omega_{n,1}|^2 dx
	\\
	& = A_p^{2} \mu_{n,1} \int_B |x|^{\a} | \omega^{-}_{n,1}|^2 dx = 	\int_B |\nabla v_p^-|^2 dx 
	\end{align*}
	Summing up, $v_p\in \mathcal N_{n,\nod}$ provided that
	\begin{equation}\label {Ap}
	A_p= \left(\mu_{n,1} \frac{\int_B |x|^{\a} | \omega^{\pm}_{n,1}|^2 dx }{\int_B |x|^{\a} | \omega^{\pm}_{n,1}|^{p+1} dx} \right)^{\frac{1}{p-1}}
	= \left(\mu_{n,1} \frac{\int_B |x|^{\a} |  \omega_{n,1}|^2 dx }{\int_B |x|^{\a} |  \omega_{n,1}|^{p+1} dx} \right)^{\frac{1}{p-1}}
	\end{equation}
	Inserting $v_p$ into \eqref{goal} gives
	\begin{align*}
	0 & <  \liminf\limits_{p\to 1}\frac{1}{\|U_{p,n}\|_{\infty}^2} \int_B |\nabla v_p|^2 dx 
	= \liminf\limits_{p\to 1}\left(\frac{A_p}{\|U_{p,n}\|_{\infty}}\right)^2 \int_B |\nabla   \omega_{n,1}|^2 dx 
	\intertext{and because $   \omega_{n,1}$ is an eigenfunction for \eqref{prima-autof-weight} related to $\mu_{n,1}$} 
	& =\liminf\limits_{p\to 1}\left(\frac{A_p}{\|U_{p,n}\|_{\infty}} \right)^2 \mu_{n,1}\int_B |x|^{\alpha}|\omega_{n,1}|^2 dx 
	\intertext{and  \eqref{Ap} gives}
	& =\liminf\limits_{p\to 1}\left(\frac{\mu_{n,1}^{\frac{p+1}{2}}}{\|U_{p,n}\|_{\infty}^{p-1}} \frac{\left(\int_B |x|^{\a} |  \omega_{n,1}|^2 dx\right)^{\frac{p+1}{2}} }{\int_B |x|^{\a} |  \omega_{n,1}|^{p+1} dx }\right)^{\frac{2}{p-1}}  .
	\end{align*}
	Since $1/(p-1)\to\infty$, a necessary condition is
	\[1 \le \liminf\limits_{p\to 1}\frac{\mu_{n,1}^{\frac{p+1}{2}}}{\|U_{p,n}\|_{\infty}^{p-1}} \frac{\left(\int_B |x|^{\a} |  \omega_{n,1}|^2 dx\right)^{\frac{p+1}{2}} }{\int_B |x|^{\a} |  \omega_{n,1}|^{p+1} dx } 
	= \frac{\mu_{n,1}}{\mu_{j,i}} \] 
	by \eqref{limite}, which implies $\mu_{j,i}\le \mu_{n,1}$.
	
	\
	
	{\it Step 3:} $\mu_{j,i}\le \mu_{0,2}$.
	\\
	We follow the same line of {\it Step 2} and define
	\begin{align*}
	v_p(x)  : = \begin{cases} A^+_p   \omega_{0,2}(x) = A^+_p \jcal_0\left(z_2(0)\, |x|^{\frac{2+\a}{2}}\right) & \text{ if } |x|\le R:=\left(\frac{z_1(0)}{z_2(0)}\right)^{\frac{2}{2+\a}} \\
	A^-_p   \omega_{0,2}(x) = A^-_p \jcal_0\left(z_2(0)\, |x|^{\frac{2+\a}{2}}\right) & \text{ if } R < |x|\le 1, \end{cases}
	\end{align*}
	and choose  the constants $A^{\pm}_p>0$ in such a way that $v_p\in {\mathcal N}^{\nod}_{p,n}$ for every $p$.
	First notice that since $v_p(x)=0$ if and only if   $|x|=R$, then  $v_p\in H^1_0(B)$  and the support of its positive/negative parts are respectively $\Omega^+=B_R$ and $\Omega^-=B\setminus B_R$.
	Next
	\begin{align*}
	\int_B |x|^{\a}|v_p^{\pm}|^{p+1} dx & =  (A^{\pm}_p)^{p+1} 	\int_{\Omega^{\pm}} |x|^{\a}|  \omega_{0,2}|^{p+1} dx ,
		\end{align*}
	\begin{align*}
	\int_B |\nabla v_p^{\pm}|^2 dx & =  (A^{\pm}_p)^2\int_B |\nabla    \omega_{0,2}^{\pm}|^2 dx = (A^{\pm}_p)^2\int_B \nabla    \omega_{0,2} \nabla    \omega^{\pm}_{0,2} dx 
	\intertext{and since $  \omega_{0,2}$ is an eigenfunction for \eqref{prima-autof-weight} related to $\mu_{0,2}$,  using $(A^{\pm}_p)^2  \omega_{0,2}^{\pm}\in H^1_0(B)$ as a test function  gives}
	& = \mu_{0,2} 	(A^{\pm}_p)^2\int_B |x|^{\a}    \omega_{0,2}     \omega^{\pm}_{0,2} dx = \mu_{0,2} 	(A^{\pm}_p)^2\int_B |x|^{\a} |    \omega^{\pm}_{0,2}|^2 dx \\
	\end{align*}
	Summing up, $v_p\in \mathcal N_{n,\nod}$ provided that
	\begin{equation}\label {Ap.2}
	A^{\pm}_p= \left(\mu_{0,2} \frac{\int_B |x|^{\a} |  \omega^{\pm}_{0,2}|^2 dx }{\int_B |x|^{\a} |  \omega^{\pm}_{0,2}|^{p+1} dx} \right)^{\frac{1}{p-1}}
	\end{equation}
	Inserting $v_p$ into \eqref{goal} gives
	\begin{align*}
	0 & <  \liminf\limits_{p\to 1}\frac{1}{\|U_{p,n}\|_{\infty}^2} \int_B |\nabla v_p|^2 dx 
	= \liminf\limits_{p\to 1}\frac{1}{\|U_{p,n}\|_{\infty}^2} \left[ (A^+_p)^2\int_B |\nabla   \omega^+_{n,1}|^2 dx +  (A^-_p)^2 \int_B |\nabla   \omega^-_{n,1}|^2 dx \right]
	\intertext{and by the previous observations}
	& = \liminf\limits_{p\to 1}\frac{\mu_{0,2}}{\|U_{p,n}\|_{\infty}^2} \left[ (A^+_p)^2\int_B |x|^{\a}|  \omega^+_{n,1}|^2 dx +  (A^-_p)^2 \int_B |x|^{\a}|  \omega^-_{n,1}|^2 dx \right]
	\intertext{next using \eqref{Ap.2} gives}	
	&= \liminf\limits_{p\to 1} 	\left(\frac{\mu_{0,2}^{\frac{p+1}{2}}}{\|U_{p,n}\|_{\infty}^{p-1}} \right)^{\frac{2}{p-1}}	\left[ 
\left(\frac{\left(\int_B |x|^{\a} |  \omega_{2,0}^+|^2 dx\right)^{\frac{p+1}{2}} }{\int_B |x|^{\a} |  \omega_{2,0}^+|^{p+1} dx }\right)^{\frac{2}{p-1}} 
		+ \left(\frac{\left(\int_B |x|^{\a} |  \omega_{2,0}^-|^2 dx\right)^{\frac{p+1}{2}} }{\int_B |x|^{\a} |  \omega_{2,0}^-|^{p+1} dx }\right)^{\frac{2}{p-1}}
		\right]
	\end{align*}
	Since $1/(p-1)\to\infty$, a necessary condition is that or 
	\[1 \le \liminf\limits_{p\to 1}\frac{\mu_{0,2}^{\frac{p+1}{2}}}{\|U_{p,n}\|_{\infty}^{p-1}} \frac{\left(\int_B |x|^{\a} |  \omega_{0,2}^{\pm}|^2 dx\right)^{\frac{p+1}{2}} }{\int_B |x|^{\a} |  \omega_{0,2}^{\pm}|^{p+1} dx } 
= \frac{\mu_{0,2}}{\mu_{j,i}} ,\] 
so that also in this case $\mu_{j,i}\le \mu_{0,2}$.
	
\

{\it Step 4:} The second eigenvalue of \eqref{prima-autof-weight} in $H^1_{0,n}$ is simple whenever $n\neq \frac{2+\a}{2}\beta$, and precisely it is given by 
$\mu_{n,1}$ or $\mu_{0,2}$ depending if $n \lessgtr \frac{2+\a}{2}\beta$.
\\
Thanks to \eqref{autov-peso-N=2} it is equivalent to see that the following inequalities hold among the zeros of different Bessel functions:
\begin{align*}
z_1(0)< z_1\left(\frac{2n}{2+\a}\right)<z_1 \left(\frac{2hn}{2+\a}\right)  \  \mbox{ as $h\ge 2$}, \\
z_2(0) \gtrless   z_1\left(\frac{2n}{2+\a}\right) \ \mbox{ if $ n \lessgtr \frac{2+\a}{2}\beta $. } 
\end{align*}
They both are consequences of the fact that the map $\beta\mapsto z_1(\beta)$ is increasing. The first one is trivial, while the second one holds true because by definition of 
$\beta$ we have $z_2(0) = z_1(\beta) \gtrless z_1\left(\frac{2n}{2+\a}\right)$ according if $\beta \gtrless \frac{2n}{2+\a}$.

\

{\it Step 5: conclusion.} 
\\
By {\it Steps 1-3} we know that $\mu_{0,1}<\mu_{j,i}\le \min\{\mu_{n,1},\mu_{0,2}\}$. Next {\it Step 4} guarantees that there are not eigenvalues in the range $\big(\mu_{0,1}, \min\{\mu_{n,1},\mu_{0,2}\}\big)$, therefore $\mu_{j,i}= \min\{\mu_{n,1},\mu_{0,2}\}= \mu_{n,1}$ if $n<\frac{2+\a}{2}\beta $, or $\mu_{0,2}$ otherwise.
So, remembering that the second eigenvalue is simple unless $n=\frac{2+\a}{2}\beta$, we have proved that  \eqref{limite} holds for $(j,i)=(n,1)$ if $n<\frac{2+\a}{2}\beta$, or else for $(j,i)=(0,2)$ when $n> \frac{2+\a}{2}\beta$.

\end{proof}

\begin{remark}\label{3.6}
	In the particular case $n=\frac{2+\a}{2}\beta$, then $z_{1}\left(\frac{2n}{2+\alpha}\right)=z_1(\beta)=z_2(0)$ by definition of $\beta$. Hence  $\mu_{n,1}=\mu_{0,2}$ has  multiplicity two as the second eigenvalue of \eqref{prima-autof-weight} in $H^1_{0,n}$, having both a radial and a nonradial eigenfunction.
	\eqref{unod-p=1-fin-1} and \eqref{unod-p=1-fin-1-alti} are equivalent and hold true, but we are not able to deduce the asymptotic behaviour of $U_{p,n}$.
	\end{remark}

\subsection{Proof of the multiplicity result}

From the asymptotic profile in Theorem \ref{n-asympt}, it is not hard to see that for $p$ close to $1$ the least energy nodal $n$-invariant solutions $U_{p,n}$ are nonradial and different one from another for $n=1, \dots  \left\lceil \frac{2+\a}{2}\beta-1\right\rceil$, radial otherwise.

First we conclude the proof of Theorem \ref{teo:existence-N=2}

	\begin{proof}[Proof of Theorem \ref{teo:existence-N=2}]
We have seen in  Corollary \ref{unod-nonradial} that there exists $\bar p=\bar p(\a)$ such that $U_{p,n}$ is nonradial for $1<p<\bar p$, as $n= 1, \dots \left\lceil\frac{2+\a}{2}\beta -1\right\rceil$.
		It remains to check that $U_{p,n}\neq U_{p,k}$ if $n\neq k$ for every $p$ in a right neighborhood of $1$,  possibly smaller than $(1,\bar p)$.
		It follows by Theorem \ref{n-asympt} which states that they converge to different eigenfunctions of \eqref{prima-autof-weight}.
	\end{proof}

After we show that the other least energy $n$-invariant solutions, i.e. $U_{p,n}$ for $n >\frac{2+\a}{2}\beta $, are radial for $p$ close to 1, by adapting to the H\'enon equation the arguments in \cite[Proposition 10.5]{GI}

\begin{proposition}\label{comeGI}
	Let $n > \frac{2+\a}{2}\beta $, then there  exists $\bar p>1$ such that for every $p\in (1,\bar p)$ $U_{p,n}$ is radial and coincides with $u^{\ast}_p$.
	\end{proposition}
\begin{proof}
	In this case we know by Theorems \ref{p1} and  \ref{n-asympt}  that both $\|u^\ast_p\|_{\infty}^{p-1}$ and $\|U_{p,n}\|_{\infty}^{p-1}$ converge to $\mu_{0,2}=\left(\frac{2+\a}{2}z_2(0)\right)^2$.
	We assume  that for a given sequence $p_k\to 1$ $U_{p_k,n} \neq u^{\ast}_{p_k} $ and deduce a contradiction. To this aim we define \[w_k = \frac{U_{p_k,n}-u^{\ast}_{p_k} }{\| U_{p_k,n} - u^{\ast}_{p_k} \|_{\infty}} .\]
	The assumption $U_{p_k,n} \neq u^{\ast}_{p_k} $  implies that there is a sequence $P_k\in B$ where $w_k(P_k)= \pm 1$, and w.l.o.g.  we can take $w_k(P_k)=1$ and $P_k\to \bar P\in \bar B$.
	Furthermore $w_k$ solves a linear Dirichlet problem 
	\begin{equation}\label{eq:diff}
	\begin{cases}
	-\Delta w_k = |x|^{\a} p_k \mu_{0,2} c_k\,  w_k & \text{ in } B , \\
	w_k=0 & \text{ on } \partial B,
	\end{cases}
	\end{equation}
where $c_k$ is given by the Mean Value Theorem
\[
c_k(x) = \frac{1}{\mu_{0,2}}\int_0^1 \left| t \, U_{p_k,n}(x) + (1-t) \, u^{\ast}_{p_k}(x) \right|^{p_k-1} dt .
\]
Clearly  $|c_k(x)|\le C \left(\|U_{p_k,n}\|_{\infty}^{p_k-1}  + \|u^{\ast}_{p_k}\|_{\infty}^{p_k-1}\right)$ is bounded, let us check that $c_k(x) \to1$ almost everywhere. Indeed the asymptotic expansion in \eqref{point-limit} for both $u^{\ast}_p$ and $U_{p,n}$ gives that
\begin{equation}\label{pass-0} 
h(x,t ): =  t \mu_{0,2}^{-\frac{1}{p_k-1}}U_{p_k,n} + (1-t) \, \mu_{0,2}^{-\frac{1}{p_k-1}}u^{\ast}_{p_k} \to  e^{c} \omega_{0,2} 
\end{equation}
uniformly for $(x,t)\in \bar B\times[0,1]$, where $c$ is the constant defined in \eqref{def:c}. 
So 
\begin{align}\label{pass-3}
c_k (x)= \int_0^1 \left|h(x,t)\right|^{p_k-1} dt 	\to 1 
\end{align}
uniformly on any closed subsect of $B$ which does not contain the zero set of $\omega_{0,2}$, i.e. the circle of radius $z_1(0)/z_2(0)$.
Therefore $w_k$ converges (weakly and then, by elliptic estimates, in $C(\bar B)$) to a solution  $w$ of \eqref{prima-autof-weight} related to $\mu_{0,2}$.
Such limit function is nontrivial since by the uniform convergence $w(\bar P)=1=\|w\|_{\infty}$, hence Lemma \ref{lem:autov-peso} yields $w(x)= \omega_{0,2} (x)=\mathcal J_0 (z_2(0)\, |x|)$.

On the other multiplying  the equation in \eqref{eq:diff} by $\omega_{0,2}$ and integrating by parts gives
\begin{align*}
p_k \mu_{0,2} \int_B |x|^{\a} c_k  w_k  \omega_{0,2} dx = \int_B \nabla w_k \nabla \omega_{0,2}  dx = \mu_{0,2} \int_B |x|^{\a}w_k  \omega_{0,2} dx .
\end{align*}
So 
\begin{equation}\label{pass-1}
(p_k-1) \int_B |x|^{\a} c_k  w_k \omega_{0,2} dx  = \int_B |x|^\a (1-c_k) w_k \omega_{0,2} dx .
\end{equation}
But using the definition of $c_k$ and the elementary equality \eqref{per-dopo} one sees that 
\begin{align*}
c_k  - 1 =  (p_k-1) \int_0^1 \log | h(t,x) |  \int_0^1  |h(t,x) |^{s(p_k-1)} ds \, dt ,
\end{align*}
which inserted into \eqref{pass-1} gives
\begin{align*}
\int_B |x|^{\a} c_k  w_k \omega_{0,2} dx  &=- \int_B |x|^\a  w_k \omega_{0,2}  \int_0^1 \log | h(t,x) |  \int_0^1  |h(t,x) |^{s(p_k-1)} ds \, dt \, dx 
\end{align*}
and passing to the limit \eqref{pass-0}, \eqref{pass-3} imply
\begin{align*}
\int_B |x|^{\a} \omega_{0,2}^2 dx  &= -\int_B |x|^\a  \omega_{0,2}^2 \left( c+\log |\omega_{0,2}|  \right) dx =0
\end{align*}
by the definition of the constant $c$ given in \eqref{def:c}.
But of course $\omega_{0,2}$ is nontrivial, and so we have reached the desired contradiction.
\end{proof}

In \cite{AG-N=2} it has been proved that for large values of $p$ the functions $U_{p,n}$ are nonradial for $n < \frac{2+\a}{2} \kappa$, where $\kappa \approx 5.1869 $ is a fixed number related to the computation of the Morse index when $p\to\infty$. Therefore  in the range $\left[\frac{2+\a}{2}\beta +1\right] \le  n \le \left\lceil \frac{2+\a}{2} \kappa-1\right \rceil$ there is a breaking of symmetry, in the sense that the curve $p\mapsto U_{p,n}$ coincides with the curve of radial solution on an interval $(1,\bar p_n)$, and then bifurcates giving rise to a global branch of nonradial solutions.
\\
Asides from $n$-invariant solutions, the issue of nonradial bifurcation from the curves $p\mapsto u_p$ of radial solutions (even with a larger number of nodal zones, and in higher dimension) and the separation of the various branches deserves a further study, which can be carried on starting from the computation of the Morse index of radial solutions at the ends of the existence range performed here and in \cite{AG-N>3}, \cite{AG-N=2}. It will be the object of a forthcoming paper \cite{Ama-bif}.

\end{document}